\documentclass{article}
\usepackage{amsmath}
\usepackage[utf8]{inputenc}
\usepackage{xcolor}
\usepackage{amsmath}
\usepackage{amsfonts}
\usepackage{amssymb}
\usepackage{bbold}
\usepackage{lmodern}
\usepackage[T1]{fontenc}
\usepackage{graphicx}
\usepackage{stmaryrd}
\usepackage{float}
\usepackage{babel}
\usepackage{amsthm}
\usepackage[linesnumbered,ruled,vlined]{algorithm2e}
\usepackage{url}
\usepackage{esint}
\usepackage{hyperref}
\usepackage{enumitem}
\usepackage[export]{adjustbox}
\usepackage{mathtools}
\usepackage{natbib}
\usepackage{geometry}

\title{Doubly nonlinear parabolic equation involving a mixed local-nonlocal operator and a convection term}
\author{Loïc Constantin\footnotemark[1]\footnote{Corresponding author. LMAP, UMR E2S UPPA CNRS 5142, Bâtiment IPRA, Université de Pau et Pays de l’Adour, Avenue de l’Université, BP 1155, 64013, Pau
Cedex, France, France. E-mail: loic.constantin@etud.univ-pau.fr}, 
Carlota M Cuesta\footnotemark[2]\footnote{University of the Basque Country (UPV/EHU), Faculty of Science and Technology, Department of Mathematics, Aptdo. 644, 48080 Bilbao, Spain. E-mail: carlotamaria.cuesta@ehu.es}}
\date{}

\newtheorem{definition}{Definition}[section]
\newtheorem{proposition}[definition]{Proposition}
\newtheorem{corollary}[definition]{Corollary}
\newtheorem{lemma}[definition]{Lemma}
\newtheorem{theorem}[definition]{Theorem}
\newtheorem{remark}[definition]{Remark}
\newtheorem{propriete}[definition]{Property}

\newcommand{\Rn}{{\mathbb{R}^d}}
\newcommand{\W}{{\mathcal{W}}}
\newcommand{\pfrac}{{(-\Delta)^s_p}}
\newcommand{\qfrac}{{(-\Delta)^s_q}}
\newcommand{\p}{{p}}

\newcommand{\Wsqz}{{W^{s,q}_0(\Omega)}}

\newcommand{\Wp}{{W^{1,\p}(\Omega)}}

\newcommand{\Wpz}{{W^{1,\p}_0(\Omega)}}

\newcommand{\dt}{{\Delta t}}
\newcommand{\ff}{{\overset{\to}{f}}}

\newcommand{\A}{{\mathcal{A}_\mu}}

\renewenvironment{abstract}
{\begin{quote}
\noindent \rule{\linewidth}{.5pt}\par{\bfseries \abstractname.}}
{\medskip\noindent \rule{\linewidth}{.5pt}
\end{quote}
}

\numberwithin{equation}{section}
\DeclareMathOperator{\Div}{div}

\begin{document}
\maketitle
\begin{abstract}
In this paper we study a doubly degenerate parabolic equation involving a convection term and the operator $\A u:=-\Delta_p u +\mu \qfrac u$ which is a linear combination of the $p$-Laplacian and the fractional $q$-Laplacian, and results in a mixed local-nonlocal nonlinear operator. The problem we study is the following, 
\begin{equation*}
    \begin{cases}
     \partial_t \beta(u)+ \A  u= \Div (\overset{\to}{f}(u))+g(t,x,u) \quad \text{in} \;Q_T:=(0,T)\times \Omega,\\
      u=0 \quad \text{in} \;  (0,T)\times (\Rn\backslash \Omega),\\
      u(0)=u_0 \text{ in } \Omega.
    \end{cases}\
\end{equation*}
We discuss existence, uniqueness and qualitative behavior of, what we call {\it weak-mild} solutions, that is weak solutions of this problem that when interpreted as $v=\beta(u)$ they are a mild solutions. In particular, we investigate stabilization to steady state, extinction and blow up in finite time and show how the occurrence of such behaviors depend on specific conditions on the nonlinearities $\beta$ (typically of porous media type), $\overset{\to}{f}$ and the source term $g$, and on their relation, in terms of certain regularity and growth conditions. 
\end{abstract}



\section{Introduction}

%Let $q\in (1,\infty),$ $p \in (2,\infty)$ and $%\mu   \geq 0$ we define 
%$$\A u=-\Delta_p u +\mu \qfrac u$$

%where $\Delta_p u=\Div(|\nabla u|^{p-2}\nabla u)$ and the fractional $q$-Laplacian is defined on the Sobolev space $\Wsqz$ up to a normalizing constant by: 

%$$\qfrac u(x)= 2 P.V. \int_\Rn \frac{|u(x)-u(y)|^{q-2}(u(x)-u(y))}{|x-y|^{d+sq}}dy.$$

In this paper, we study the problem:
\begin{equation}\label{pbp}\tag{P}
    \begin{cases}
     \partial_t \beta(u)+ \A  u= \Div (\overset{\to}{f}(u))+g(t,x,u) \quad \text{in} \;Q_T:=(0,T)\times \Omega,\\
      u=0 \quad \text{in} \;  (0,T)\times (\Rn\backslash \Omega),\\
      u(0,\cdot)=u_0(\cdot) \text{ in } \Omega,
    \end{cases}\
\end{equation}
where $\Omega\subset \Rn$ is a bounded domain with smooth boundary. Here, for $q\in (1,\infty)$, $p \in (2,\infty)$ and $\mu  \geq 0$ we define the mixed local-nonlocal nonlinear operator 
$$
\A u=-\Delta_p u +\mu \qfrac u,
$$
where $\Delta_p u=\Div(|\nabla u|^{p-2}\nabla u)$ is the $p$-Laplacian and $\qfrac u$ denotes the fractional $q$-Laplacian. The latter is defined on the Sobolev space $\Wsqz$ up to a normalizing constant by: 
$$
\qfrac u(x)= 2 P.V. \int_\Rn \frac{|u(x)-u(y)|^{q-2}(u(x)-u(y))}{|x-y|^{d+sq}}dy.
$$

%ADD REFS FOR p-L and f q-L

%ADD on motivation for the parabolic problwm with driving term...

We shall impose different growth conditions on the nonlinearities $\beta$, $\ff$ and $g$, see Section \ref{secmain} for more details. We point out that, in any case, we take $\beta$ such that $\beta(0)=0$ and increasing, and $\ff(0)=0$, so that no hyperbolic regions are created and the boundary conditions are well suited (see \cite{alibaud} for the case $q=2$ and no $p$-Laplacian). 

In this paper we prove the existence and uniqueness of a weak and mild solution of the problem, in the sense that we explain below in Definition~\ref{wmildsol}. We then show results on the qualitative behavior of the solution for both linear as well as power type source terms.

The study of equations involving nonlocal operators has recently attracted considerable interest due to their role in physical phenomena with long-range interactions, prompting the creation of a specialized mathematical theory. These equations, encompassing elliptic and parabolic problems, emerge in fields like finance, Lévy type stochastic processes, physics, populations dynamics, and fluid dynamics. For instance models using nonlocal operators overcome some limitations of continuum mechanics, not being well suited for multi-phase materials or fracturing (see e.g. \cite{drapaca,silling,zimmermann}). 

In such models the presence of external forcing convection terms (due to, for example, injection of fluids at non negligible rates or gravity effects) might be considered. Mathematically, this term alone might not be supported unless a more regular term, such a diffusion term, is present. In this regard, to ensure the well-posedness of the general model where the diffusion is given by the fractional $q$-Laplacian as above, we have included the $p$-Laplacian term, thus keeping solutions well-defined.  

On the other hand, mixed local and nonlocal operators and, in particular, a linear combination of a $p$-Laplacian and a fractional $q$-Laplacian has seen an increase in mathematical interest too. For example, some results on the existence, uniqueness and regularity of solutions of an elliptic problem have been obtained in \cite{antonini, opmixte, filippis, garain}). One can also find some results on the semilinear parabolic problem with such a mixed operator in \cite{biagi}.

The particular case $\ff=0$, is the case of a degenerate doubly parabolic equation, and has been studied separately for the fractional $p$-Laplacian as well as for the classical $p$-Laplacian (see e.g. \cite{CollierHauer, constantin2024existence,Giacomoni,kato,uchida}). One can also find some results on the fractional $p$-Laplacian evolution equation, that is the problem involving the nonlocal operator with one nonlinearity. For example it has been solve for the homogeneous case (e.g. \cite{mazon,vazquez}) as well as for the nonhomogeneous case (e.g. \cite{abdellaoui,talibi}).

On the other hand, the case $\beta=Id$ and with $\mu=0$ or with, formally, $s=1$, gives a local convection diffusion equation. The associated elliptic equation has been studied in e.g. \cite{faraci,faria, motreanu}. For the parabolic convection diffusion equation we can see \cite{ji, jin,nakao} where, like in this article, the authors take a convection term of the form $b(u).\nabla u$ in unbounded domains. 

The local doubly degenerate equation without a source term but with a convection term, has been studied, for example in \cite{ShangCheng}, where the authors study this doubly degenerate parabolic equation for a gradient source term of the form $|\nabla u^q |^\sigma$ on the whole space to find existence and non-existence results of weak solutions. 

One can also see \cite{laptev} where the author finds weak solutions of a second-order (thus local, $\mu=0$) quasilinear parabolic equation with a double nonlinearity and a source term of the form $A_0(t,x,u,Du)$. The main difference with our existence result are the conditions we set on the nonlinear terms $A_0$ and $\beta$ as well as the methods used.

In this article we treat the questions of existence, uniqueness and qualitative behavior such as extinction, blow up and stabilization of solutions of the problem \eqref{pbp}. In this regard, we generalize the results of \cite{constantin2024existence}. As in \cite{constantin2024existence} we will use the accretive operator theory combined with a discretization in time scheme to prove existence and uniqueness of a combined weak and mild solution, that we call {\it weak-mild} solution.

More precisely, for a nonlinear source term, we obtain, by a discretization in time method, local and global in time existence of weak solutions (that is solutions satisfying a variational formulation), as well as an $L^1$-contraction property. By combining this contraction property, given by the notion of mild solution, and Gronwall's Lemma, we get uniqueness under some local Lipschitz condition on the source term (see Section \ref{sec2}).
For this approach we first need to study the associated elliptic problem and prove accretivity of the operator $A$, given by $Av=\A  (\beta^{-1}(v))-\Div (\overset{\to}{f}(\beta^{-1}(v)))$ (defined in Section \ref{sec0}), 
as well as some density of its domain $D(A)$ (see Section \ref{acc}).
We finish by showing qualitative behavior of these solutions (see Section \ref{sec3}). More precisely, using a comparison principle as well as a subsolution method we show a stabilization result, or convergence in time to the stationary solution for $\mu>0$ and for a linear source term. Finally, using energy methods we show extinction and blow up for a power type source term. We notice  that up to our knowledge, even in the specific case $\mu=0,\beta=Id$, that is the case of a parabolic equation involving the $p$-Laplacian and a convection term, no energy method has been used to prove blow up in a bounded domain. 

In the next section we give some preliminary results that are needed in our analysis. In Section~\ref{secmain} we list precisely our results after giving the conditions on the nonlinearities. The proofs of these results follow in the next sections, as we outline also in Section~\ref{secmain}. To make the reading more fluent, some auxiliary technical results are given in the Appendix. 

\subsection{Preliminaries}\label{sec0}

In this section we set the functional setting and define the notion of solution that we will use later. We start with the functional setting, for additional information we refer for instance to \cite{rando,bisci}. 

We start by recalling the definition of the fractional Sobolev space $W^{s,q}(R^d)$:

$$
W^{s,q}(\Rn)=\bigg\{ u\in L^q(\Rn) \; | \; \int_\Rn \int_\Rn  \frac{|u(x)-u(y)|^q}{|x-y|^{d+sq}}dy dx<\infty \bigg\},
$$

endowed with the natural norm:

$$
\|u\|_{W^{s,q}(\Rn)}=\bigg (\int_\Rn |u|^qdx +  \int_\Rn \int_\Rn  \frac{|u(x)-u(y)|^q}{|x-y|^{d+sq}}dy dx \bigg)^\frac{1}{q}.
$$

The space $\Wsqz$ is defined by $\Wsqz= \big \{ u\in W^{s,q}(\Rn) \; | \; u=0 \text{ on } \Rn \backslash \Omega \big \},$ endowed with the Banach norm

$$
\|u\|_\Wsqz = \left( \int_\Rn \int_\Rn  \frac{|u(x)-u(y)|^q}{|x-y|^{d+sq}}dy dx\right)^{\frac1q}.  
$$

The space $\Wsqz$ is a reflexive space and a Poincar\'e inequality (see {\it e.g.} \cite[Th. 6.5]{bisci}) provides the equivalence of the two norms $\|\cdot \|_{W^{s,q}(\Rn)}$ and $\|\cdot\|_\Wsqz$ on this space.

For $sq<d$ and $\gamma\in [1,\frac{dq}{d-sq}]$, we have the continuous embedding 
$\Wsqz\hookrightarrow L^\gamma(\Omega)$ and if $\gamma<\frac{dq}{d-sq}$ the embedding is compact (see {\it e.g.} \cite[Coro. 7.2]{rando}). 

For $sq\geq d$, we have the compact embedding $\Wsqz\hookrightarrow L^ \gamma(\Omega)$ for all $\gamma\in [1,+\infty)$ (see {\it e.g.} \cite{rando}).

We denote by $\Wpz$ the classical Sobolev space and recall that we have similar continuous and compact embedding results for this space (see e.g. \cite{brezis}).

%\vspace{0.3cm}

By definition of the fractional $q$-Laplacian and its symmetry, we have that for any $u$, $v \in \Wsqz$ (rearranging the integral and applying Fubini):
$$
\langle\qfrac u,v\rangle=\int_\Rn \int_\Rn \frac{|u(x)-u(y)|^{q-2}(u(x)-u(y))(v(x)-v(y))}{|x-y|^{d+sq}}dx dy.
$$
We define for $u,v\in \Wpz\cap \Wsqz$:
$$
\langle \A u ,v \rangle =\int _\Omega |\nabla u|^{p-2}\nabla u. \nabla v+ \mu \langle\qfrac u,v\rangle .
$$
Here and in what follows $\langle\cdot, \cdot \rangle$ denotes the dual product, where we shall not specify the dual space it corresponds to, since it will be clear from the context.

We recall that $\A$ is the Gateau-differential of the convex functional:

$$
J_\A (v):= \frac{1}{p}\|\nabla v\|^p_p +\frac{\mu}{q}\|v\|^q_\Wsqz.
$$

For simplicity of notation, we set 
$$\W:= 
\begin{cases}
   \Wpz \cap \Wsqz, &  \text{ if } \mu >0,  \\
    \Wpz, & \text{ if } \mu=0,
\end{cases}
%\W:= \Wpz \cap \Wsqz  \text{ if } \mu >0 
%\text{ and } \W :=\Wpz \text{ if } \mu=0,
$$ 
and denote its dual space by $\W^*$. The intersection space $\Wpz \cap \Wsqz$ is endowed with the norm $\| \cdot\|_{\Wpz}+\|\cdot\|_\Wsqz$. We then notice that if $J_\A (v)<\infty$ then $ \|v\|_\W<\infty$. 

We also define the space:
$$
X_T:=\{ v\in  L^\infty(Q_T)\cap L^\infty(0,T;\W)\; | \; \partial_t  v \in L^2(Q_T) \}.
$$

We shall show the existence of weak solutions in the following sense:
\begin{definition}[Weak solution]\label{defsol2}
A function $u \in X_T$ is a weak solution of \eqref{pbp} if $\beta(u)\in C([0,T];L^{2}(\Omega))$, $u(0,\cdot)=u_0$ {\it a.e.} in $\Omega$ and such that for any $t\in [0, T]$:
\begin{equation}\label{fv2}
\begin{split}
\bigg[\int_\Omega \beta(u) \varphi\,dx\bigg]^t_0-\int_0^t\int_\Omega \beta(u) \partial_t \varphi \,dxd\tau+&\int_0^t \left(\langle\A u,\varphi\rangle -\langle \Div \ff (u),\varphi\rangle  \right)\,d\tau\\
&=\int_0^t\int_{\Omega}g(\tau,x,u)\varphi\,dxd\tau,
\end{split}
\end{equation}
for any $\varphi \in H^1(0,T;L^2(\Omega))\cap L^1(0,T;\W)$. 
\end{definition}

We also define the notion of mild solutions of \eqref{pbp} as in \cite{barbu2}. This is done first for the following problem, for which, formally, $v=\beta(u)$:

\begin{equation*}\label{pbo}\tag{$P_0$}
    \begin{cases}
     \partial_t v + A({v}) = h &  \text{in} \;Q_T,\\
      v=0 & \text{in} \;(0,T) \times (\Rn\backslash \Omega), \\
      v(0,\cdot)=v_0&  \text{in} \; \Omega,
    \end{cases}\
\end{equation*} 
where $h \in L^1(\Omega)$ and the operator $A : L^1(\Omega) \to L^1(\Omega)$ is thus defined by: 
\begin{equation}\label{defA}
A (v) = \A(\beta^{-1}(v))-\Div \ff (\beta^{-1}(v)),
\end{equation}
on the domain:
\begin{equation}\label{domA}
D(A)=\{v\in L^1(\Omega) \; | \;  \beta^{-1}(v)\in \W, \; A(v) \in L^1(\Omega)\}.
\end{equation}

We next define mild solutions of problem \eqref{pbo} as in \cite[Def. 4.3.]{barbu2}, and then mild solutions of problem \eqref{pbp}. For completeness, we recall the definition of $\varepsilon$-approximate solutions in Definition \ref{defapp} of the Appendix.
\begin{definition}[Mild solutions]\label{mildsol}
For $T>0$ and $h\in L^1(Q_T)$, we say that a mild solution of the problem \eqref{pbo} is a function $v\in  C([0,T];L^1(\Omega))$ such that for any $\varepsilon>0$, there is an $\varepsilon$-approximate solution $U$ that satisfies $\|U-u\|_{L^1(\Omega)}<\varepsilon$ for all $t\in[0,T]$.

We say that $u$ is a mild solution of problem \eqref{pbp} if $v=\beta(u)$ is a mild solution of \eqref{pbo} for $h=g(t,x,\beta^{-1}(v))$.
\end{definition}

We will search for solutions that are weak and mild solutions. Namely, as in \cite{constantin2024existence}, we define the notion of weak-mild solution: 
\begin{definition}[Weak-mild solutions]\label{wmildsol}
Let $T>0$, a weak-mild solution $u$ of \eqref{pbp} is a weak solution $u$ such that it is also a mild solution of \eqref{pbp}.

We call global weak-mild solution to a function $u\in L^\infty_{loc}(0,\infty;L^\infty(\Omega))$ such that $u$ is a weak-mild solution of \eqref{pbp} on $Q_{ T}$ for any $ T<\infty$.
\end{definition}

Throughout this manuscript, we use the following conventions. First, for a given a real valued function $f$, we use the notation 
\[
f_+=\max\{0,f\} \text{ and } f_-=\min\{0,f\},
\]
so that $f=f_+ + f_-$ and $|f|=f_+ -f_-$. In general, we use the letters $u$ and $v$ to denote solutions to problem \eqref{pbp} and to \eqref{pbo}, respectively. But for simplicity in notation (e.g. in order to avoid indexes), we might use both when comparing two solutions of either \eqref{pbp} or \eqref{pbo}.

\subsection{Main results}\label{secmain}
Before giving the main results, we set the conditions on the nonlinearites and other parameters that we assume throughout this paper. 

We take $u_0\in \W\cap L^\infty(\Omega)$. We set $\overset{\to}{f}=(f_1, \ldots ,f_d)$ such that $f_i\in W^{1,\infty}_{loc}(\mathbb{R})$ and $f_i(0)=0$. We take $\beta: \mathbb{R} \to \mathbb{R}$ increasing bijective and locally $\alpha$-Hölder continuous for $\alpha\in (0,1]$, with $\beta(0)=0$ and such that:
\begin{align*}
&\bullet \forall K>0,\;\beta(t)-\beta(t') \geq C_K(t-t') \text{ for } -K\leq t'\leq t\leq K, \label{beta1}\tag{$\beta_1$} \\
&\bullet |\beta(t)| \leq c(1+|t|^{p-1}). \label{beta2}\tag{$\beta_2$}
\end{align*}
For example, we can take $\beta(t)=|t|^{\frac{1}{m}-1}t$ for $m\geq 1$. 

We take $g$ to be a Carath\'eodory function. We shall set the condition, for given constants $c_g$, $q_g\in (0,\infty)$:
\begin{equation}\label{growcond}\tag{g1}
|g(t,x,s)|\leq c_g(1+|s|^{q_g}) \text{ a.e. in } Q_T.
\end{equation}

We are now ready to state the existence results.

\begin{theorem}\label{para1}
Let $\beta$ be an odd function that satisfies \eqref{beta1} and \eqref{beta2}. Then, we have:\\

If \eqref{growcond} is satisfied, then there exists $T>0,$ such that there exists $u$ a weak-mild solution of \eqref{pbp}.\\

Or if, instead, the following condition on $g$ is satisfied,

\begin{equation}\label{g2}\tag{g2}
\lim_{R\to \infty} \frac{\sup_{[0,R]}g}{\beta(R)}=C_g,
\end{equation}
then, there exists a global weak-mild solution.

\end{theorem}
Imposing more restrictions on $\ff$ and $g$, we have the following Theorem, that gives global existence of bounded solutions.
\begin{theorem}\label{para2}
Let $g(x,t,s)=g(s)$ be Lipschitz function with $g(0)=0$. Assume that $u_0\in [0,k]$ a.e. and $f_i\in W^{1,\infty}([0,k])$ with $f_i=0$ on $\mathbb{R}\backslash ]0,k[$ for some $k>0$ fixed.

Then, there exists $u\in [0,k]$ a.e. that is a global weak-mild solution of \eqref{pbp}.
\end{theorem}

In order to get uniqueness we add the following condition on $g$. 

For any $T$, $R>0$, there exists a positive constant $C(T,R)$ such that, for all $(t,x)\in Q_T$, and all $s_1$, $s_2\in B(0,R)$,
\begin{equation}\label{gLip}\tag{g3}
|g(t,x,s_1)-g(t,x,s_2)\leq C(T,R)|\beta (s_1)- \beta (s_2)|.
\end{equation}
Then, we can prove the following:
\begin{theorem}[Uniqueness]\label{prop:uniqueness}
Let $T>0$, if $g$ satisfies \eqref{gLip}, there is a unique weak-mild solution of \eqref{pbp}.

Moreover, if $u$ and $v$ are two weak-mild solutions of \eqref{pbp} with right-hand side $g_1$ and $g_2$, respectively, and initial data $u_0$ and $v_0\in { \W\cap L^{\infty}(\Omega)}$. Then, the following $L^1$-contraction property holds:
\begin{equation}\label{L1contr}
\sup_{[0,T]}\|\beta(u)-\beta(v)\|_{L^1(\Omega)}\leq \|\beta(u_0)-\beta(v_0)\|_{L^1(\Omega)}+\int_0^T\|g_1(t,x,u)-g_2(t,x,v)\|_{L^1(\Omega)}\,dt.
\end{equation}
\end{theorem}

\begin{remark}
In relation to the conditions on $g$, we can think of power law type functions. 
For example, if we take for $m>1$, $\beta(s)=|s|^{\frac{1}{m}-1}s$ and $g(u)=|u|^{r-1}u$, then we have that the condition $\eqref{g2}$ is equivalent to $\frac{1}{m}\geq r$, and the condition \eqref{gLip} is equivalent to $\frac{1}{m}\leq r$.
\end{remark}

We end this section by stating the qualitative behavior of our solutions. We start with the convergence to the steady state.

\begin{theorem}[Stabilization]\label{thstab}
Let $\mu>0$, $g=h\in L^\infty(\Omega)$ be a given nonnegative function. We denote by $u_{stat}$ the unique nonnegative solution of the stationary problem \eqref{pbstab}. If $u_0\in[0,u_{stat}]$ then, as $t\to \infty$, $u(t)\to u_{stat}$ in $L^\gamma(\Omega)$ for all $\gamma<\infty$, where $u$ is the global solution of Theorem \ref{para1}.
\end{theorem}
For the next two Theorems we take for $m>1$ and $r>0$
\begin{equation}\label{powerlaw}
g(t,x,u)=|u|^{r-1}u \text{ and  }\beta (s) = |s|^{\frac{1}{m}-1}s.
\end{equation}
We shall use energy estimates in the proofs. We start with the extinction result.
\begin{theorem}[Extinction]\label{thextin}
Let $g$ and $\beta$ be given by \eqref{powerlaw}, and let $\mu >0$ and $q<r+1<\frac{1}{m}+1$. Then, finite time extinction occurs for $\|u_0\|_{k+\frac{1}{m}}$ small enough with $k\geq \min(1,\frac{d-sq-d(q-1)m}{msq})$, in the sense that there exists $T_e$ such that $u(t)=0$ for $t\geq T_e$.%$T_e>0$ such that $\|u(t)\|_{k+\frac{1}{m}}\to 0$ as $t\to T_e$.
\end{theorem}

For the blow up result, we first define the functional:
\begin{equation}\label{eq:energy}
E(u)= \frac{\mu}{q}\| u\|^q_\Wsqz+ \frac{1}{p}\| \nabla u\|^p_p-\frac{1}{r+1} \| u \| ^{r+1}_{r+1}.
\end{equation}
Then, we can prove the following:
\begin{theorem}[Blow up]\label{thexpl}
%We set
%$$
%E(u)= \frac{\mu}{q}\| u\|^q_\Wsqz+ \frac{1}{p}\| \nabla u\|^p_p-\frac{1}{r+1} \| u \| %^{r+1}_{r+1}.
%$$
Assume that either 
\begin{itemize}
    \item $r>p-1$ and $\mu=0$ or,
    \item $r>\min(p-1,q-1)$ and $\mu>0$, 
\end{itemize}
and that $\ff$ satisfies:
\begin{equation}\label{condf}\tag{f1}
|f_i'(s)|\leq c(1+|s|^{\gamma}) \text{ with } 2(\gamma +1)<p.
\end{equation}
Then, given $u_0\in \W \cap L^\infty(\Omega)$ that satisfies
\begin{equation}\label{ineqener}
E(u_0)<-\max_{x>0} (c(x^2+x^{2(\gamma +1)})-\| u_0\| ^{r-\frac{1}{m}}_{\frac{1}{m}+1}x^p),
\end{equation}
where $c$ is a constant depending only on $\ff,p,r,m,d, \Omega$, then the weak-mild solution $u$ blows up for $T\leq T^*:=\tilde c \| u_0\|^{\frac{1}{m}-r}_{\frac{1}{m}+1}$ where $\tilde{c}$ depends only on $\ff,p,r,m,d, \Omega$. I.e. there exists $T_b>0$ such that $\|u(t)\|_{1+\frac{1}{m}}\to \infty$ as $t\to T_b$.
\end{theorem}

%?¿? ADD the sense of blow up and extinction, it means, that $\| u\|_{\frac{1}{m}+1} to \infty$ as $t\toT^-$ and that $\| u\|_{\frac{1}{m}+k} \to 0$ as $t\to T^-$

\begin{remark}
Using the conditions on $r$ we have $r>\frac{1}{m}$ and with $2(\gamma+1)<p$ the condition \eqref{ineqener} is attainable by taking $Ku_0$ for $K$ big enough when $r+1>\max(q,p)$. 
\end{remark}

In the rest of the paper we proceed as follows. In Section~\ref{acc}, we study the operators $\A$ and $A$. We show accretivity of $A$ and density of the domain $D(A)$ in order to later use the results of \cite{barbu2}. We also study the elliptic problem associated to \eqref{pbp}. In Section \ref{sec2} we use the results of Section \ref{acc} to prove Theorems \ref{para1}, \ref{para2} and \ref{prop:uniqueness}. In Section \ref{sec3} we address qualitative behavior, first by proving Theorem \ref{thstab} then by proving Theorems \ref{thextin} and \ref{thexpl}.

\section{The operator and the elliptic problem}\label{acc}

In this section we show preliminary results on the operator and the elliptic problem, later we will use these results to get a weak solution of \eqref{pbp} and to be able to use \cite[Th. 4.2]{barbu2} and work with the mild solutions theory.

We take in this Section $\ff$ such that $f_i\in W^{1,\infty}(\mathbb{R})$. We will be able to use the results of this section for $f_i\in W_{loc}^{1,\infty}(\mathbb{R})$, since the solution we find in the next section are bounded independently of $\ff$.

\begin{theorem}\label{thacc}
$A$, as defined in \eqref{defA}, is accretive in $L^1(\Omega)$.
\end{theorem}

\begin{proof}
As in \cite[Th. 3.5]{barbu2} with $\beta^{-1}$ increasing we have that the operator $u\mapsto \A(\beta^{-1}(u))$ is accretive in $L^1(\Omega)$.

We now show that $u\mapsto -\Div \ff (\beta^{-1}(u))$ is accretive in $L^1(\Omega)$, that is, for $u$, $v\in D(A)$:
$$
-\int_\Omega (\Div \ff (\beta^{-1}(u))-\Div \ff (\beta^{-1}(v))) sgn(u-v)\geq 0.
$$

As $sgn(u-v)=sgn(\beta^{-1}(u)-\beta^{-1}(v))$ it is enough to show that:

$$
\int_\Omega \partial_{x_i}(f_i(\beta^{-1}(u))-f_i(\beta^{-1}(v)))sgn(\beta^{-1}(u)-\beta^{-1}(v))=0, \text{ for } i=1,\dots, d,
$$
or, equivalently, where without loss of generality we use the same notation as above, we show that for $u,v\in \W$,
$$
\int_\Omega \partial_{x_i}(f_i(u)-f_i(v))sgn(u-v)=0 \text{ for } i=1,\dots,d.
$$

For each $i=1,\dots,d$ we set $\overline{f}_i$ and $\underline{f}_i$, to be the non-decreasing and non-increasing, respectively, part of $f_i$, defined by:
$$
\overline{f}_i'=(f_i')_+ , \; \underline{f}_i' = (f_i')_- \text{ and }\overline{f}_i(0)=\underline{f}_i(0)=f_i(0).
$$
%where $\psi=\psi_++\psi_-$, such that $\overline{f}_i,\underline{f}_i$ are respectively non-decreasing and non-increasing. 
We have $\partial_{x_i} f_i(u) = f_i'(u) \partial_{x_i} u =(\overline{f}_i'(u)+\underline{f}_i'(u))\partial_{x_i} u $ a.e. and then
\begin{equation*}
\begin{split}
\int_\Omega & \partial_{x_i}(f_i(u)-f_i(v))sgn(u-v),\\
 &=\int_\Omega \partial_{x_i}(\overline{f}_i(u)-\overline{f}_i(v)+\underline{f}_i(u)-\underline{f}_i(v))sgn(u-v), \\
&  =\int_\Omega \partial_{x_i}(\overline{f}_i(u)-\overline{f}_i(v))sgn(u-v)+\partial_{x_i}(\underline{f}_i(u)-\underline{f}_i(v))sgn(u-v).
\end{split}
\end{equation*}
By monotonicity, $sgn(u-v)=sgn(\overline{f}_i(u)-\overline{f}_i(v))=-sgn(\underline{f}_i(u)-\underline{f}_i(v))$, and this implies that
\begin{equation*}
\begin{split}
\int_\Omega \partial_{x_i}(\overline{f}_i(u)-\overline{f}_i(v))sgn(u-v) &= \int_\Omega \partial_{x_i}(|\overline{f}_i(u)-\overline{f}_i(v)|),\\
& =\int _\Omega \nabla( (\overline{f}_i(u)-\overline{f}_i(v))_+-(\overline{f}_i(u)-\overline{f}_i(v))_-). e_i.
\end{split}
\end{equation*}

We have 
$$
\nabla f_i(u) = f_i'(u) \nabla u \text{ a.e.},
$$
that, with $f_i\in  W^{1,\infty}(\mathbb{R})$, gives $f_i'(u) \in L^{p'}(\Omega)$. Thus, we have $f_i(u), f_i(v)\in W^{1,p }(\Omega)$ and then $ (\overline{f}_i(u)-\overline{f}_i(v))_+ \in W^{1,p}_0(\Omega)$, which implies
$$
\int _\Omega \nabla (\overline{f}_i(u)-\overline{f}_i(v))_+. e_i=-\int _\Omega  (\overline{f}_i(u)-\overline{f}_i(v))_+ \Div( e_i)=0.
$$

Similarly, we obtain $\int_\Omega \partial_{x_i}(\underline{f}_i(u)-\underline{f}_i(v))sgn(u-v)=0$ and hence
$$
\int_\Omega (-\Div \ff (\beta^{-1}(u))+\Div \ff (\beta^{-1}(v))) sgn(u-v)= 0.
$$

We can then conclude that $A$ is accretive in $L^1(\Omega)$.
\end{proof}

We look at the elliptic problem associated to \eqref{pbo}:

%\begin{minipage}{0.25 \textwidth}
%\begin{equation*}
%\begin{cases}
%      v + \lambda A v = h  \quad \text{in } \Omega,\\
%      v=0 \quad \text{on } \Rn \backslash \Omega,
%    \end{cases}
%\end{equation*}
%\end{minipage}
%\begin{minipage}{0.15\textwidth}
%or equivalently
%\end{minipage}
%\begin{minipage}{0.4\textwidth}
%\begin{equation}\tag{Q}\label{pbel}
%    \begin{cases}
%    \beta(u) + \lambda  \A u = \Div \ff(u) + h  \quad \text{in } \Omega,\\
%    u=0 \quad \text{on } \Rn\backslash \Omega.
%    \end{cases}
%\end{equation}
%\end{minipage}

\begin{equation}\tag{Q}\label{pbel}
\begin{cases}
      v + \lambda A v = h  \quad \text{in } \Omega,\\
      v=0 \quad \text{on } \Rn \backslash \Omega,
    \end{cases}
\text{ or, equivalently, }
\begin{cases}
    \beta(u) + \lambda  \A u = \Div \ff(u) + h  \quad \text{in } \Omega,\\
    u=0 \quad \text{on } \Rn\backslash \Omega.
    \end{cases}
\end{equation}
%\end{minipage}

\begin{definition}
We say that $u:=\beta^{-1}(v)$ is a weak solution of \eqref{pbel}, if $u \in \W$ and:
$$
\int_\Omega \beta(u) \varphi+ \lambda\langle \A  u, \varphi \rangle= \langle h + \Div (\overset{\to}{f}(u)),\varphi \rangle
$$
for all $\varphi \in \W$.
\end{definition}

\begin{theorem}\label{exielli}
Let $\lambda >0$ and $h\in \W^*$, then problem \eqref{pbel} admits a weak solution.
\end{theorem}

\begin{proof}
We will solve \eqref{pbel} using a minimization method and Schafer's fixed point Theorem. 

Let $h\in \W^*$ and $w\in \Wpz$, first we solve
\begin{equation}\label{35bis}
    \begin{cases}
      \beta(u) + \lambda \A u = \lambda\Div \ff(w) + h  \quad \text{in} \;Q_T,\\
      u=0 \quad \text{on} \;\Rn\backslash \Omega.
    \end{cases}\
\end{equation}
 
Since $f_i\in W^{1,\infty}(\mathbb{R})$, we have $\ff(w)\in (L^{p'}(\Omega))^d$ and thus also $\Div \ff (w)\in W^{-1,p'}(\Omega)$. And, since $p>2$, we have that for all $\varphi \in \Wpz$:
$$
\int_\Omega \Div (\overset{\to}{f}(w)) \varphi =-\int_\Omega \nabla \varphi.\overset{\to}{f}(w).
$$

We define $J_w$ on $\Wpz$ by:
$$
J_w(u)=\int_\Omega B(u)+\lambda J_\A (u)-\int_\Omega  h u +\lambda \int_\Omega \nabla u. \overset{\to}{f}(w),
$$

where $B(t)=\int_0^t \beta(s) ds$. We then have $  B(t)\geq 0$, therefore
$$
J_w(u) \geq \frac{\lambda }{p} \|\nabla u\|^p_{L^p(\Omega)}
+\lambda \frac{\mu}{q}\|u\|_\Wsqz^q
-\|h\|_{L^\infty(\Omega)}\|u\|_{L^1(\Omega)}-\lambda \|\overset{\to}{f}\|_{L^\infty(\Omega)}\|\nabla u\|_{L^1(\Omega)},
$$
and then $J_w$ is coercive in $\W$.

Let $(u_n)_{n\in\mathbb{N}}$ be a minimizing sequence of $J_w$, then $(u_n)_n$ is bounded in $\W$. The spaces are reflexive so (up to a subsequence) $u_n \rightharpoonup u$ as $n\to\infty$ in $\Wpz$ and in $\Wsqz$. The norm is a weakly lower semicontinuous functional, and by compact embedding results $u_n \to u$ in $L^{p}(\Omega)$ as $n\to\infty$. Then, there exists $g\in L^{p}(\Omega)$ such that up to a subsequence $|u_n|\leq g$ and $u_n\to u$ a.e. as $n\to\infty$. Then, also, $B(x,u_n) \to B(x,u)$ a.e., and for all $n$
$$
|B(x,u_n)| \leq C_1+C_2|u_n|^{p}\leq C_1 +C_2|g|^{p}.
$$
Applying now the dominated convergence theorem, we get $\int_\Omega B(x,u_n) \to \int_\Omega B(x,u)$ as $n\to\infty$. 

Summarizing, when $n\to \infty$ we obtain that $u$ is a minimizer of $J_w$ and that
$$
\int_\Omega\beta(u)\varphi+\lambda  \langle \A u,\varphi\rangle=\int_\Omega (\lambda \Div (\overset{\to}{f}(w))+h)\varphi,
$$
for all $\varphi$ in $\W$. 

We now define the fixed point map $\Gamma:w \mapsto u$ on $\Wpz$ where $u$ is the solution of \eqref{35bis}. Next we prove the existence of a fixed point by showing that the hypotheses of Schafer`s Theorem are satisfied in our case.

We set $u_i=\Gamma(w_i)$ for $i=1,2$. We subtract the variational formulations corresponding to each $i=1,2$ and we take the test function $\varphi= (u_1-u_2)$. We have with \eqref{ineg2} a lower bound on the left handside of the resulting equation:
\begin{equation}\label{36}
\int_\Omega (\beta(u_1)-\beta(u_2))(u_1-u_2)+\lambda \langle \A u_1 -\A u_2 , u_1-u_2 \rangle \geq c \|u_1-u_2\|^p_\Wpz.
\end{equation}

For the divergence term on the right handside we have, with the Hölder inequality, the lower bound:
\begin{equation}\label{38}
\lambda \int_\Omega (\ff(w_1)-\ff(w_2)).\nabla(u_1-u_2)\leq \lambda \|u_1-u_2\|_\Wpz \|\ff(w_1)-\ff(w_2)\|_{p'}.
\end{equation}

By combining \eqref{36} and \eqref{38} we have:
\begin{equation}\label{39}
\|u_1-u_2\|^{p-1}_\Wpz \leq C \|\ff(w_1)-\ff(w_2)\|_{p'}.
\end{equation}

We can now show that $\Gamma$ is continuous and compact. Let $w_n$ be bounded in $\Wpz$ then, up to a subsequence, $w_n \to w$ in $L^\gamma(\Omega)$ for all $\gamma<p^*$ as $n\to \infty$, where $p^*$ is the critical Sobolev exponent. Using the regularity of $f$ we have $f_i(w_n) \to f_i(w) $ in $L^{p'}(\Omega).$ Also using \eqref{39} we have $u_n \to u$ in $\Wpz$ which gives $\Gamma$ continuous and compact in $\Wpz$.

We now prove that
$$
\{u\in \Wpz \; | \; u=\tau \Gamma(u) \text{ for } \tau\in [0,1 ]\},
$$
is bounded. Taking the test function ${u}/{\tau}$ we have:
$$
\lambda \bigg\|\frac{u}{\tau}\bigg\|^{p-1}_\Wpz \leq \|h\|_{W^{-1,p'}(\Omega)}+C \|\ff(u)\|_{p'},
$$
and the assumption $f\in W^{1,\infty}(\mathbb{R})$ gives a uniform bound in $\Wpz$. Thus we can apply Schafer's fixed point Theorem to conclude that there exists a solution. 

\end{proof}

As a consequence we have:
\begin{corollary}
If $p>d$ or $\mu>0$ and $qs>d$ then $ L^1(\Omega)\hookrightarrow \W^*$ and $A$ is maximal.
\end{corollary}

We now prove the density of the domain.
\begin{lemma}\label{lemacc}
We have $\beta^{-1}(v)\in \W\cap L^\infty(\Omega)$ implies $v \in \overline{D(A)}^{L^1(\Omega)}$.
\end{lemma}

\begin{proof}
Let $v$ such that $\beta^{-1}(v)\in \W$, using \eqref{beta2} we have $v\in L^{p'}(\Omega)\hookrightarrow \W^*$. Let $\varepsilon>0$ we define $v_\varepsilon  \in \W$ such that
$$
v_\varepsilon-v+\varepsilon Av_\varepsilon=0.
$$
Using Theorem \ref{exielli} and \eqref{domA} we have that $v_\varepsilon \in D(A)$ is well defined. We now take the test function $\varphi=\beta^{-1}(v_\varepsilon)-\beta^{-1}(v)$ in the weak formulation, which, by using also the convexity of $J_\A$ and that $f_i\in W^{1,\infty}(\mathbb{R})$, gives:
\begin{equation*}
\begin{split}
\int_\Omega (\beta^{-1}(v_\varepsilon)-\beta^{-1}(v))(v_\varepsilon-v)+&\varepsilon J_\A (\beta^{-1}(v_\varepsilon)),\\
&\leq \varepsilon (J_\A(\beta^{-1}(v))+\|f\|_\infty\|\nabla (\beta^{-1}(v_\varepsilon)-\beta^{-1}(v))\|_p).
\end{split}
\end{equation*}

Now, since $\beta^{-1}(v)\in \W$, $\beta$ increasing and $J_\A(u)\geq \frac{\|\nabla u\|^p_p}{p}$, we get that $\beta^{-1}(v_\varepsilon)$ is bounded in $\W$. And using compact embedding results as well as \eqref{beta2}, we get that $v_\varepsilon \to v$,  as $\varepsilon\to 0$, in $L^1(\Omega)$, up to a subsequence.
\end{proof}

We now show the following convergence result:
\begin{proposition}\label{convlap}
Let $(u_n)_n$ be a bounded sequence in $L^\infty(0,T;\W)$, such that $u_n \overset{*}{\rightharpoonup}u \text{ in }L^\infty(0,T;\W)$ as $n\to\infty$ and that% up to a subsequence and:
\begin{equation}\label{34}
\int_0^T\langle \A u_n, u_n-u \rangle\to 0, \text{ as } n\to\infty.
\end{equation}
Then, up to a subsequence, $\nabla u_n \to \nabla u$ in $(L^p(Q_T))^d$ and $\int_0^T\langle\A u_n - \A u,\varphi\rangle\to 0$ for all $\varphi\in L^1(0,T;\W)$ as $n\to \infty$.
\end{proposition}

\begin{proof}

We first show that $\qfrac u_n \to \qfrac u$ as $n\to \infty$. Since $(u_n)_n$ is bounded in $L^\infty(0,T;\W)$, then
$$ 
U_n: (t,x,y) \mapsto \frac{|u_n(t,x)-u_n(t,y)|^{q-2}}{|x-y|^{(d+sq)/q'}}(u_n(t,x)-u_n(t,y)),
$$
is bounded in $L^\infty(0,T;L^{q'}(\mathbb{R}^{2d}))$ uniformly in $n$. The compact embedding gives almost everywhere convergence and we identify the limit as $n\to\infty$:
$$ 
U_n\overset{*}{\rightharpoonup} \frac{|u(t,x)-u(t,y)|^{q-2}}{|x-y|^{(d+sq)/q'}}(u(t,x)-u(t,y)) \text{ in } L^\infty(0,T;L^{q'}(\mathbb{R}^{2N})),
$$
which gives $\int_0^T\langle \qfrac u_n - \qfrac u,\varphi\rangle\to 0$ as $n\to\infty$ for all $\varphi\in L^1(0,T;\Wsqz)$. 

We now show the convergence of $\Delta_p$. Using the weak convergence of $(u_n)_n$ we have that $\int_0^T\langle \A u ,u_n -u\rangle$ goes to $0$ when $n\to +\infty$, then with \eqref{34}:
\begin{equation*}
\int_0^T\langle \A u_n -\A u ,u_n -u\rangle \to 0 \text{ as }n\to \infty.
\end{equation*}

On the other hand, by definition:
\begin{equation}\label{40}
\langle \A u_n -\A u ,u_n -u\rangle   = \langle -\Delta_p u_n +\Delta_p u ,u_n -u\rangle 
+\mu \langle \qfrac u_n -\qfrac u ,u_n-u\rangle,
\end{equation}
and since $\qfrac$ and $-\Delta_p$ are accretive in $L^2(\Omega)$, 
\begin{equation}\label{24}
\int_0^T\langle -\Delta_p u_n +\Delta_pu ,u_n -u\rangle \to 0 \text{ as } n\to\infty.
\end{equation}
Combining \eqref{24} with \eqref{ineg2} we obtain $\nabla u_n \to \nabla u $ in $(L^p(Q_T))^d $. 

We next show that $|\nabla u_n|^{p-2}\nabla u_n\to |\nabla u|^{p-2}\nabla u$ in $(L^{p'}(Q_T))^d$. Indeed, using \eqref{ineg1} and the Hölder inequality, we have:
\begin{equation*}
\begin{split}
& \int_{Q_T} \left||\nabla u_n| ^{p-2} \nabla u_n-|\nabla u| ^{p-2} \nabla u\right|^{\frac{p}{p-1}}dx, \\
& \leq C\int_{Q_T} |\nabla (u_n-u)|^{\frac{p}{p-1}}(|\nabla u_n|+|\nabla u|)^{\frac{p(p-2)}{p-1}}dx ,\\
& \leq C\left\||\nabla u_n-\nabla u|^{\frac{p}{p-1}}\right\|_{L^{p-1}(Q_T)}\left\|(|\nabla u_n|+|\nabla u|)^\frac{p(p-2)}{p-1}\right\|_{L^{\frac{p-1}{p-2}}(Q_T)}.
\end{split}
\end{equation*}
Since $(u_n)_n$ is bounded in $L^\infty(0,T;\Wpz)$ we obtain 
\begin{equation}\label{42}
|\nabla u_n|^{p-2}\nabla u_n \to |\nabla u|^{p-2}\nabla u \text{ in }(L^{p'}(Q_T))^d .
\end{equation}
Finally, using the $L^\infty(0,T;\Wpz)$ bound, the weak convergence and the identification given by \eqref{42} we get $\int_0^T\langle -\Delta_p u_n + \Delta_p u,\varphi\rangle\to 0$ as $n\to\infty$ for all $\varphi\in L^1(0,T;\Wpz)$.
\end{proof}

\section{Proof of the existence results}\label{sec2}
In this section we will prove Theorems \ref{para1}, \ref{prop:uniqueness} and \ref{para2}, in this order.

\subsection{Proof of Theorems \ref{para1} and \ref{prop:uniqueness}}

We assume in this section the hypotheses of Theorem \ref{para1}. In particular using that $\beta$ is an odd function we have $\beta(|u|)=|\beta(u)|$. As we shall solve the parabolic problem using a discretization scheme, we first solve the elliptic problem:
\begin{equation}\label{pbe1}
    \begin{cases}
     \frac{1}{\dt} \beta(u)+ \A  u= h+\Div (\overset{\to}{f}(u)) \quad \text{in} \;\Omega,\\
      u=0 \quad \text{in} \; \Rn\backslash \Omega.
    \end{cases}\
\end{equation}
We have the following:
\begin{lemma}\label{elli2}
Let $h\in L^\infty(\Omega)$ and $ \dt>0$, then, there exists a weak solution of \eqref{pbe1} $u\in \W$ such that $\|\beta(u)\|_\infty\leq \dt\|h\|_{\infty}$.
\end{lemma}

\begin{proof}
We define $R= \beta^{-1}(\dt \|h\|_{L^\infty(\Omega)})$ and for $i=1,\dots,d$, {$\tilde{f}_i\in W^{1,\infty}(\mathbb{R})$} such that:
\begin{align*}
\tilde{f}_i(t)= \left\{ \begin{array}{lr}  f_i(R) \text{ if } t\geq R, \\
f_i(t) \text{ if } |t|\leq R,\\
 f_i(-R) \text{ if } t\leq -R, \\
\end{array}\right.
\end{align*}
and $\|\tilde{f}_i\|_{L^\infty(\mathbb{R})}\leq \|f_i\|_{L^\infty(-R,R)}$. We also note that $\tilde{f}_i$ is Lipschitz. We set $\ff_R=( \tilde f_1,\dots,\tilde f_d)$. Using Theorem \ref{exielli} we have a weak solution $u_R$ of \eqref{pbe1} with $\ff$ replaced by $\ff_R$.

We now show that $|u_R|\leq R$. Indeed, we write,
$$
\frac{1}{\dt} \beta(u_R)-\|h\|_{\infty}+\A  u_R=h+\Div (\overset{\to}{f}_R(u_R))-\|h\|_{\infty},
$$
and take the test function $ \varphi =\dt(u_R-R)_+\in \Wpz$. Observe that,
$$
\langle \A u_R,(u_R-R)_+\rangle \geq 0 ,
$$
and since $\dt\|h\|_{\infty}=\beta(R)$, we can write
\begin{equation*}
\int_\Omega (\beta(u_R)-\beta(R))(u_R-R)_+\leq \int_\Omega \dt (u_R-R)_+(h+\Div (\overset{\to}{f_R}(u_R))-\|h\|_{\infty}).
\end{equation*}
Now, we observe that $\Div \ff_R=0$ on $[R,\infty)$, and this gives
\begin{equation*}
\int_\Omega (\beta(u_R)-\beta(R))(u_R-R)_+ 
\leq \int_\Omega  \dt(u_R-R)_+( |h|-\|h\|_{\infty})
 \leq 0.
\end{equation*}
Finally, since $\beta$ is increasing, we have $u_R\leq R$.\\
Similarly, we can add $\|h\|_{\infty}$ to both sides of the equation and take the test function $(u_R+R)_-$, this finally gives $|u_R|\leq R$ and by the de definition of $R$, $|\beta(u_R)|\leq \dt (\|h\|_{\infty})$. Then, we have that $u_R$ is in fact a solution of \eqref{pbe1}, thus we write $u=u_R$, and satisfies $\|\beta(u)\|_\infty\leq \dt \|h\|_{\infty}$. 
\end{proof}

We are now ready to solve the parabolic problem. Note that the next proof is using a discretization method similar to that in e.g. \cite{constantin2024existence}.
\begin{proof}[Proof of Theorem \ref{para1}]

We divide the proof into 4 steps. In steps 1, 2 and 3 of the proof we solve: 
\begin{equation}\label{pbpr}\tag{$P_R$}
    \begin{cases}
     \partial_t \beta(u)+ \A  u= \Div (\overset{\to}{f}(u))+g_R(t,x,u) \quad \text{in} \;Q_T,\\
      u=0 \quad \text{in} \;(0,T)\times(\Rn\backslash \Omega) ,\\
      u(0)=u_0 \text{ in } \Omega,
    \end{cases}\
\end{equation}
where 
\begin{equation}\label{gr}\tag{$g_R$}
g_R(t,x,\theta)=g(t,x,sgn(\theta)\min(|\theta|,R)) \text{ for }R>0.
\end{equation}
In step 4 we show that the solution is a mild solution, as well as that it satisfies $ |u|\leq R$, and in fact $g_R(u)=g(u)$.

\vspace{0.3cm}
\textbf{\underline{Step 1: Time-discretization scheme}}

Let $T>0$ and, $N\in \mathbb{N}^*$, we set $\dt=\frac{T}{N}$, $t_n=n\dt$. And we introduce, for $\gamma\in [1,+\infty)$,  the linear continuous operator $T_\dt$ from $L^\gamma(Q_T)$ to itself, defined by  
\begin{equation*}
T_\dt \psi(t,x):=\displaystyle{\fint^{t_n}_{t_{n-1}}}\psi(\tau,x)d\tau\ \text{ for } (t,x)\in [t_{n-1},t_n)\times \Omega.
\end{equation*} 
We let $u^0=u_0$ and for $1\leq n\leq N$, we define $u^n\in \Wpz$ as the weak solution of
\begin{equation*}
    \begin{cases}
     \frac{1}{\dt} \beta(u^n)+ \A  u^n=g^n+\Div (\overset{\to}{f}(u^n))+\frac{1}{\dt} \beta(u^{n-1}) \quad \text{in} \;\Omega,\\
      u^n=0 \quad \text{in} \; \Rn\backslash \Omega,
    \end{cases}\
\end{equation*}
where $g^n=g_\dt (t_{n-1})$ for $g_\dt=T_\dt g_R(\cdot,\cdot,u^{n-1})$. %Let $T>0$ and, $N\in \mathbb{N}^*$, we set $\dt=\frac{T}{N}$, $t_n=n\dt$, for $1\leq n\leq N$, we define $u^n\in \Wpz$:
%\begin{equation*}
%    \begin{cases}
%     \frac{1}{\dt} \beta(u^n)+ \A  u^n=g^n+\Div (\overset{\to}{f}(u^n))+\frac{1}{\dt} \beta(u^{n-1}) \quad \text{in} \;\Omega,\\
%      u^n=0 \quad \text{in} \; \Rn\backslash \Omega,
%    \end{cases}\
%\end{equation*}
%where $g^n(x):=\fint^{t_n}_{t_{n-1}} g_R(\tau ,x,u^{n-1}(x)) d\tau$.

By Theorem \ref{elli2} the sequence $(u^n)_n$ is well defined. With it, we define the following functions over $[0,T]$:
\begin{itemize}
    \item  $u_\dt =u^n, \text{ on }[t_{n-1},t_n[$,
    \item $ \widetilde \beta(u_\dt)=\dfrac{t-t_{n-1}}{\dt}(\beta(u^n)-\beta(u^{n-1}))+\beta(u^{n-1})\text{ on }[t_{n-1},t_n[$,
    \item $ \tilde{u}_\dt=\dfrac{t-t_{n-1}}{\dt}(u^n-u^{n-1})+u^{n-1}\text{ on }[t_{n-1},t_n[$.
\end{itemize}
Then
$$
\partial_t \widetilde \beta(u_\dt)+ \A u_\dt=g_\dt+\Div(\overset{\to}{f}(u_\dt)),
$$
for all $t\in (0,T)$ in the weak sense.% of Theorem \ref{elli}. 

\vspace{0.3cm}

\textbf{\underline{Step 2: A priori Estimates}}

We first show that $u_\dt$ is uniformly bounded. Using Lemma~\ref{elli2} for $h=g^n+\frac{1}{\dt}\beta(u^{n-1})$, we have that,
$$
\|\beta(u^n)\|_\infty\leq \dt \|g^n\|_{\infty}+\|\beta(u^{n-1})\|_\infty,
$$
which gives
$$
\|\beta(u_\dt)\|_\infty\leq \|\beta(u_0)\|_\infty+\sum^N_{n=1} \dt \|g^n\|_{\infty}.
$$
The assumption \eqref{growcond} implies $|g^n|\leq c_g(1+R^q)$, which applied to the above inequality and with the a.e. convergence gives
\begin{equation}\label{bound}
\|\beta(u)\|_\infty  \leq  \|\beta(u_0)\|_\infty+ T c_g(1+R^q).
\end{equation}

We notice that using this uniform bound \eqref{bound} and \eqref{beta1}, we can say that there exists $C>0$ such that for all $n,m$:
\begin{equation}\label{1}
|\beta(u^n)-\beta(u^m)|\geq C |u^n-u^m|,
\end{equation}
where $C$ is independent of $\dt$.

We now take the test function $\varphi=u^n-u^{n-1}$ in the weak formulation of each $u^n$ problem, thus we have
\begin{equation}\label{26}
\begin{split}
\int_\Omega & \frac{1}{\dt} (\beta(u^n)-\beta(u^{n-1}))(u^n-u^{n-1})+\langle \A  u^n,u^n-u^{n-1}\rangle\\
&= \int_\Omega (g^n+\Div (\overset{\to}{f}(u^n)))(u^n-u^{n-1}).
\end{split}
\end{equation}
Using Young's inequality, \eqref{1} and the convexity of $J_\A$, we then get
\begin{equation*}
\begin{split}
C_0&\frac{1}{\dt}\|u^n-u^{n-1}\|^2_{L^2(\Omega)}+J_\A(u^n)\\
&\leq J_\A(u^{n-1})+C_1\dt(\|\Div\overset{\to}{f}(u^n)\|_{L^2(\Omega)}^2+ \|g^n\|_{L^2(\Omega)}^2).
\end{split}
\end{equation*}
Now passing, in this inequality, to the sum for any $m\leq N$ and the definition of $g_R$ \eqref{gr}, gives:
\begin{equation*}
\begin{split}
C_0&\frac{1}{\dt}\sum_{n=1}^{m}\|u^n-u^{n-1}\|^2_{L^2(\Omega)}+J_\A(u^{m}) \\
&\leq J_\A(u^{0})+\dt\sum_{n=1}^{m}\left(C_1\|\Div\overset{\to}{f}(u^n)\|_{L^2(\Omega)}^2+C_R \right).
\end{split}
\end{equation*}
Since $m$ is arbitrary above, and using $\frac{1}{p}\|v\|^p_\Wpz \leq J_\A (v)$, we arrive at
\begin{equation}\label{12}
\begin{split}
C_0& \|\partial_t \tilde{u}_\dt\|^2_{L^2(Q_T)}+\frac{1}{p}\|u_\dt\|^p_{L^\infty(0,T;\Wpz)}\\
&\leq J_\A(u_0)+C_RT +C_1\|\Div\overset{\to}{f}(u_\dt)\|_{L^2(Q_T)}^2.
\end{split}
\end{equation}

The last term on the right handside of \eqref{12} can be further estimated from above, using Theorem \ref{chain} to get $\partial_{x_j}(f_i(u))=f_i'(u)\partial_{x_j} u$ even if $f_i\in W^{1,\infty}_{loc}(\mathbb{R})$, and by using the uniform bound \eqref{bound}, then
\begin{equation}\label{14}
\|\Div(\overset{\to}{f}(u_\dt)\|_{L^2(Q_T)}^2\leq C(1+\|\nabla u_\dt\|_{L^2(Q_T)}^2),
\end{equation}
where $C$ is independent of $\dt$.

%\vspace{0.3cm}

Combining \eqref{12} and \eqref{14}, we then obtain

\begin{equation*}
\|u_\dt\|^p_{L^\infty(0,T;\Wpz)}\leq  c \left(1+\|\nabla u_\dt\|_{L^2(Q_T)}^2\right)\leq C\left(1+\|u_\dt\|^2_{L^\infty(0,T;\Wpz)}\right),
\end{equation*}
and the assumption $p>2$ gives that $u_\dt$ is bounded in $L^\infty(0,T;\W)$. Now, coming back to \eqref{14}, this implies,
\begin{equation}\label{13}
\|\Div(\overset{\to}{f}(u_\dt)\|_{L^2(Q_T)}^2\leq C.
\end{equation}

Finally, with \eqref{26}, \eqref{12} and \eqref{13} we have that
\begin{equation}\label{2}
\|\partial_t \tilde{u}_\dt\|^2_{L^2(Q_T)}+\|J_\A(u_\dt)\|_{L^\infty(0,T)}\leq \tilde{C},
\end{equation}
and
\begin{equation}\label{7}
 \sum^{N}_{n=1} \frac{1}{\dt} \int_\Omega (\beta(u^{n})-\beta(u^{n-1}))(u^{n}-u^{n-1})\leq \tilde{C},
\end{equation}
where $\tilde{C}$ is independent of $\dt$.

\vspace{0.3cm}

\textbf{\underline{Step 3: Convergence}}\\
We now employ the estimates obtain in the previous step to obtain a limit and to show that we can pass to the limit in the different terms of the equation.

We start by obtaining the convergence of the solution sequences. Using \eqref{2}, we have that, up to a subsequence, there exists $u_1$, $u_2\in L^\infty(0,T;\W)$ such that
\begin{equation}\label{3}
u_\dt , \widetilde{u}_\dt \overset{*}{\rightharpoonup}u_1, u_2 \text{ in }L^\infty(0,T;\W) \text{ as } \dt \to 0,
\end{equation}
and
\begin{equation}\label{4}
\partial_t \tilde{u}_\dt \rightharpoonup \partial_t u_2  \text{ in } L^2(Q_T) \text{ as } \dt \to 0.
\end{equation}

With \eqref{2}, and the Aubin-Simon Lemma \cite{simon2}, we obtain
$$
\tilde{u}_\dt \to u_2 \text{ in } C([0,T];L^2(\Omega))  \text{ as } \dt \to 0.
$$
This, together with the interpolation inequality and \eqref{bound}, imply that
\begin{equation}\label{27}
\tilde{u}_\dt \to u_2 \text{ in } C([0,T];L^\gamma(\Omega)), \; \forall \gamma \in [1,\infty) \text{ as } \dt \to 0.
\end{equation}
Now, from $(\partial_t \widetilde{u}_\dt)_\dt$ being bounded in $L^2(Q_T)$, we also get that
\begin{equation}\label{11}
\begin{split}
C\geq \sum_{n=1}^{N} \dt \int_\Omega \left( \frac{u^n-u^{n-1}}{\dt}\right)^2 & \geq \max_n\int_\Omega(u^n-u^{n-1})^2\dt^{-1} \\
& \geq \dt^{-1} \sup_{t\in (0,T)}\|(\tilde{u}_\dt- u_\dt)(t)\|_{L^2(\Omega)}^2
\end{split}
\end{equation}
from which we deduce that $u_1=u_2$ a.e.. From now on, we just use the notation $u$ for $u_1$ and $u_2$. 

Using \eqref{27} combined with \eqref{11}
$$
u_\dt,u_\dt(\cdot-\dt) \to u \text{ in }L^\infty(0,T;L^2(\Omega)) \text{ as } \dt \to 0, 
$$
and using again the interpolation inequality we deduce that for any $\gamma\in [1, +\infty)$
\begin{equation}\label{18}
u_\dt,u_\dt(\cdot-\dt) \to u \text{ in } L^\infty(0,T;L^\gamma(\Omega)) \text{ as } \dt \to 0.
\end{equation}

We can now pass to the limit in the source term. First, by applying the dominated convergence theorem, we have that for any $\gamma\in [1,+\infty)$
\begin{equation}\label{35}
g_R(\cdot,\cdot,u_\dt(\cdot-\dt)) \to g_R(\cdot,\cdot,u) \text{ in } L^\gamma(Q_T) \mbox{ as }  \dt\to 0.
\end{equation}

Using that the operator $T_\dt$ is continuous in $L^\gamma(Q_T)$, that $g_\dt=T_\dt g_R(\cdot,\cdot,u^{n-1})$ and that $T_\dt \psi$ tends to $\psi \text{ in } L^{\gamma}(Q_T)$ as $\dt \to 0$, we get that 
 \begin{equation*}
\begin{split}
\|g_\dt -&g_R(\cdot,\cdot,u)\|_{L^\gamma(Q_T)}\\
&\leq \|g_\dt -T_\dt g_R(\cdot,\cdot,u)\|_{L^\gamma(Q_T)}+ \|g_R(\cdot,\cdot,u) -T_\dt g_R(\cdot,\cdot,u) \|_{L^\gamma(Q_T)}\\ 
& \leq \|T_\dt\|\|g_R(\cdot,\cdot,u^{n-1}) - g_R(\cdot,\cdot,u)\|_{L^\gamma(Q_T)}+ \|g_R(\cdot,\cdot,u) -T_\dt g_R(\cdot,\cdot,u) \|_{L^\gamma(Q_T)}.
\end{split}
\end{equation*}
Thus, for any finite $\gamma\geq 1$, $g_\dt$ tends to $g_R(\cdot,\cdot , u)$ in $L^\gamma(Q_T)$ as $\dt\to 0$.

%\vspace{0.3cm}

Now we prove the convergence for $(\beta(u_\dt))_{\dt}$. We have \eqref{18} and that $\beta$ is locally Hölder continuous, then for $\gamma\in [1,+\infty)$ 
\begin{equation}\label{22}
\beta(u_\dt)\to \beta (u)\text{ a.e. and in }L^\infty(0,T;L^\gamma(\Omega)) \text{ as } \dt \to 0.
\end{equation}

Again, since $\beta$ is $\alpha$-Hölder continuous and using \eqref{7}, we obtain
\begin{equation*}
C  \geq \sum^{N}_{n=1} \frac{1}{\dt} \int_\Omega (\beta(u^{n})-\beta(u^{n-1}))(u^{n}-u^{n-1})
\geq \max_n \frac{1}{\dt} \int_\Omega |\beta(u^{n})-\beta(u^{n-1})|^{ \frac{1}{\alpha}+1},
\end{equation*}
which implies
\begin{equation}\label{21}
\|\widetilde \beta (u_\dt)-\beta(u_\dt)\|_{L^\infty(0,T;L^{\frac{1}{\alpha}+1}(\Omega))}\to 0 \text{ as } \dt\to 0.
\end{equation}

On the other hand, from \eqref{22}, \eqref{21} and the interpolation inequality, we have that for any $\gamma \in (1,+\infty)$
\begin{equation}\label{16}
\widetilde \beta (u_\dt)\to \beta (u)\text{ in }C([0,T];L^\gamma(\Omega)) \text{ as } \dt \to 0.
\end{equation}

For the divergence term, we use that $\overset{\to}{f}$ is Lipschitz, then we can simply conclude that, for any $\gamma\in[1,+\infty)$,
\begin{equation}\label{6}
\overset{\to}{f}(u_\dt) \to \overset{\to}{f}(u) \text{ in } L^\infty(0,T;L^\gamma(\Omega)) \text{ as } \dt \to 0.
\end{equation}

We now want to prove the convergence of $(\A(u_\dt))_\dt$ using Proposition~\ref{convlap} and the weak formulation of the equation with the test fucntion $\varphi=u_\dt - u$.

 %\begin{equation}\label{28}
%\int_0^t \langle\qfrac(u_\dt),\varphi\rangle \to \int_0^t  \langle \qfrac u,\varphi\rangle \text{ as } \dt \to 0.
%\end{equation}
First, we show that $\int_{Q_T} \partial_t\widetilde\beta(u_\dt)(u_\dt-u)\to 0$ as $\dt\to 0$. Indeed, by Lemma~\ref{ipp}, we can write
\begin{equation}\label{25}
\begin{split}
\int_{Q_T} \partial_t\widetilde\beta(u_\dt)(u_\dt-u)&=\int_{Q_T} \partial_t\widetilde\beta(u_\dt)(u_\dt-\tilde{u}_\dt+\tilde{u}_\dt-u),\\
&=-\int_{Q_T} \widetilde\beta(u_\dt)\partial_t(\tilde{u}_\dt-u)+\Big[\int_\Omega\widetilde\beta(u_\dt)(\tilde{u}_\dt-u)\Big]_0^T\\
&\hspace{0.4cm}+\int_{Q_T} \partial_t\widetilde\beta(u_\dt)(u_\dt-\tilde{u}_\dt).
\end{split}
\end{equation}
Now, using that $|u_\dt-\tilde{u}_\dt|\leq |u^n-u^{n-1}|$ on $[t_{n-1},t_n)$ and \eqref{7}, we get for the last term in \eqref{25} that 
\begin{equation}\label{19}
%\begin{split}
%\int_{Q_T} \partial_t\widetilde\beta(u_\dt)(u_\dt-\tilde{u}_\dt)&\leq \sum_{n=1}^N \dt \int_\Omega \frac{|\beta(u^n)-\beta(u^{n-1})|}{\dt}|u^n-u^{n-1}|\\
 %& \leq  C\dt.
%\end{split}
\int_{Q_T} \partial_t\widetilde\beta(u_\dt)(u_\dt-\tilde{u}_\dt)\leq \sum_{n=1}^N \dt \int_\Omega \frac{|\beta(u^n)-\beta(u^{n-1})|}{\dt}|u^n-u^{n-1}|
 \leq  C\dt.
\end{equation}
This means, that combining \eqref{4}, \eqref{16}, \eqref{19} and \eqref{25} we get
\begin{equation}\label{20}
\int_{Q_T} \partial_t\widetilde\beta(u_\dt)(u_\dt-u)\to 0 \text{ as } \dt \to 0.
\end{equation}

We then test the weak formulation with the function $\varphi =u_\dt-u$. This implies, using  \eqref{3}, \eqref{18},  \eqref{35}, \eqref{6} and \eqref{20}, that
\begin{equation}\label{23}
\begin{split}
\int_0^T\langle \A u_\dt ,u_\dt -u\rangle=&-\int_{Q_T} \partial_t\widetilde\beta(u_\dt)(u_\dt-u)\\
&+\int_{Q_T}g_\dt(u_\dt-u)-\int_{Q_T}\overset{\to}{f}(u_\dt(t-\dt)).\nabla(u_\dt-u)\\
&\longrightarrow 0.
\end{split}
\end{equation}

We are now in the position to apply Proposition \ref{convlap}, to get
\begin{equation}\label{41}
\nabla u_\dt \to \nabla u \text{ in }(L^p(Q_T))^d \text{ as } \dt \to 0,
\end{equation}
and, for all $\varphi \in L^1(0,T;\W)$, 
$$
\int_0^T \langle \A u_\dt , \varphi \rangle \to \int_0^T \langle \A u, \varphi \rangle \text{ as } \dt \to 0.
$$

We can now pass to the limit $\dt \to 0$ in the equation. We take $\varphi \in L^1(0,T;\W)\cap H^1(0,T;L^2(\Omega))$, and using Lemma \ref{ipp} we finally get,
\begin{equation*}
\begin{split}
-&\int_0^t \int_\Omega \widetilde\beta(u_\dt) \partial_t \varphi+\Big[\int_\Omega\varphi \widetilde\beta(u_\dt)\Big]^t_0+\int_0^t\langle \A u_\dt,\varphi\rangle\\
&=\int_0^t\int_\Omega (-\overset{\to}{f}(u_\dt). \nabla\varphi)+g_\dt \varphi,
\end{split}
\end{equation*}

for all $\dt$ and $t\leq T$. And taking $\dt\to 0$ we obtain a solution $u\in X_T$ of \eqref{pbpr} satisfying the variational formulation for all $\varphi \in L^1(0,T;\W)\cap H^1(0,T;L^2(\Omega))$ and such that $u(0)=u_0$ a.e.. 

%\vspace{0.3cm}

\textbf{\underline{Step 4: $u$ is a weak-mild solution of \eqref{pbp}}}

We first show that $u$ is weak-mild solution of \eqref{pbp} for $g_R$. Let $\varepsilon>0$, using $g_\dt \to g_R$ and the results of the previous steps we have that for $\dt$ small enough, $\beta (u_\dt)$ is an $\varepsilon$-approximation solution (see Definition \ref{defapp}). Then $\beta(u)$ is a mild solution for the data $g_R(t,x,u)$.

%\vspace{0.3cm}

%We now that $g_R=g$ under some condition on $T.$ 
By Step 2 and \eqref{growcond} we have that
$$
\|\beta(u)\|_\infty  \leq  \|\beta(u_0)\|_\infty+ T c_g(1+R^{q_g}).
$$
Using that $\beta$ is increasing, we can choose $R$ big enough as well as $T$ small enough, to get 
$$
\|u\|_\infty \leq \beta^{-1}( \|\beta(u_0)\|_\infty+ T c_g(1+R^{q_g}))\leq R,
$$
and this gives $|u|\leq R$, and hence $g_R(t,x,u)=g(t,x,u)$. Thus we have that $u$ is weak-mild solution of \eqref{pbp}. 

%\vspace{0.3cm}

We now show that if \eqref{g2} holds then we have a global weak-mild solution. Using \eqref{g2} for all $T\in(0,C_g^{-1})$ and $u_0\in \W\cap L^\infty(\Omega)$, there exists $R>0$ big enough such that,
$$
\|\beta(u_0)\|_\infty + T \sup_{[0,R]} g \leq \beta(R),
$$
which gives a solution. 

We set $T>0$, then we have a solution $u_1$ over $Q_T$ and $u_1(T)\in L^\infty(\Omega) \cap \W$. By induction we build $(u_n)_n$ weak-mild solution of \eqref{pbp} such that $u:= u_{n+1}(\cdot-nT) $ on $[nT,(n+1)T]\times \Omega$ is a global weak-mild solution.
\end{proof}

\begin{remark}
As we get the solution $|u|\leq C$ where $C$ is independent of $\ff$, we can take $\ff_C\in W^{1,\infty}(\mathbb{R})$ such that $\ff_C=\ff$ in $B(0,C)$ and we can work using the results of Section \ref{acc}.
\end{remark}

%\vspace{0.3cm}

\begin{remark}\label{rk}
We have that $A$ is accretive (via Theorem \ref{thacc}) and $\beta(u_0)\in \overline{D(A)}^{L^1(\Omega)}$ (via Lemma \ref{lemacc}), which gives by \cite[Th. 4.1]{barbu2} uniqueness of the mild solution for $g:=h\in L^\infty(Q_T)$. Thus all weak-mild solution are the limit of the discretized solutions.
\end{remark}

%\begin{proposition}
%Let $T>0$ and let $u$, $v$ be two weak-mild solutions of \eqref{pbp} with right-hand side respectively $g=h_1,h_2\in {L^\infty(Q_T)}, $ and initial data $u_0,\,v_0\in { \W\cap L^{\infty}(\Omega)}$. Then, we have 
%
%\begin{equation}\label{L1contr}
%\sup_{[0,T]}\|\beta(u)-\beta(u)\|_{L^1(\Omega)}\leq \|\beta(u_0)-\beta(v_0)\|_{L^1(\Omega)}+\int_0^T\|h_1-h_2\|_{L^1(\Omega)}\,dt.
%\end{equation}
%\end{proposition}

\begin{proof}[Proof of Theorem \ref{prop:uniqueness}]
We first prove the $L^1$-contraction properties for $h_1$, $h_2\in L^\infty(Q_T)$. Using Remark \ref{rk}, we can consider the solutions obtained by the proof Theorem \ref{para1}. Let $u_{\dt}$, $v_{\dt}$ be the discretized solutions defined associated with $(h_1,u_0)$ and $(h_2,v_0)$, respectively.

From the $L^1$-accretivity of $A$, we get:
\begin{equation*}
\begin{split}
\|\beta(u^{n})-\beta(v^{n})\|_{L^1(\Omega)} & \leq \|\beta(u^{n-1})-\beta(v^{n-1})\|_{L^1(\Omega)}+\dt\|h_1^n-h_2^n\|_{L^1(\Omega)},\\
& \leq \|\beta(u_0)-\beta(v_0)\|_{L^1(\Omega)} +\| (h_1)_\dt-  (h_2)_\dt\|_{L^1(Q_T)}.
\end{split}
\end{equation*}

Passing to the limit as $\dt\to0$, we find \eqref{L1contr} for $h_1$ and $h_2\in L^\infty(Q_T)$ as source terms. On the other hand, using the uniqueness of the weak-mild solution for a fixed $h$ we just have to take $h_1=g_1(u)$ and $h_2=g_2(v)$ to obtain \eqref{L1contr}.

%\vspace{0.3cm}

Uniqueness of the weak-mild solution when \eqref{gLip} holds, now follows from \eqref{L1contr} and Gronwall's Lemma (see e.g. \cite[Lemma 4.2.1, p. 55]{c&h}).

\end{proof}

\subsection{Proof of Theorem \ref{para2}}

We assume in this section the hypotheses of Theorem \ref{para2}. That is, let $k>0$, $u_0\in \W$, $u_0\in [0,k]$ a.e. and $f_i\in W^{1,\infty}([0,k])$ and $f_i=0$ in $\mathbb{R}\backslash ]0,k[$. 

We notice that in this section the condition on $\beta$ being odd is not necessary.
Instead, using \eqref{beta1} we have that $\lambda\beta(t)-g(t)$ is nondecreasing on $\mathbb{R}$ for $\lambda>0$ big enough. As we use a time discretization method where the role of $\lambda$ is played by $1/\dt$, we can assume set such a large value of $\lambda$ in what follows in order to solve the corresponding elliptic problem. 

We first solve the following elliptic problem for $g\in C^{0,1}(\mathbb{R})$ such that $g=0$ on $\mathbb{R}\backslash (0,k)$:
\begin{equation}\label{pbe}
    \begin{cases}
      \lambda\beta(u)+ \A u=\Div (\overset{\to}{f}(u))+g(u)+\lambda \beta(\tilde{u}) \quad \text{in} \;\Omega,\\
      u=0 \quad \text{on} \; \Rn\backslash \Omega,
    \end{cases}\
\end{equation}
where $\lambda \beta(\tilde{u})\in L^\infty(\Omega)$ for some $\tilde{u}\in [0,k]$ a.e..

\begin{theorem}\label{elli}
There exists $u\in [0,k]$ that is a weak solution of \eqref{pbe}.
\end{theorem}

\begin{proof}
The proof is similar to the proof of Theorem \ref{elli2}. Indeed, using a fixed point method, that $g$ is Lipschitz and $g(0)=0$, and that $\lambda\beta(t)-g(t)$ is nondecreasing on $\mathbb{R}$, is enough to obtain a solution. 

We now prove that $u\in [0,k]$ a.e.. We start by shoing $u\geq 0$, in order to do that we take the test function $u_-= u\mathbb{1}_{\{u\leq 0\}}\in \W$ and observe that $\langle \qfrac u ,u_- \rangle \geq 0$. This gives,
$$
\int_\Omega (\lambda\beta(u)-g(u))u_-+ \|\nabla u_-\|^p_{L^p(\Omega)}\leq -\int_\Omega \overset{\to}{f}(u).\nabla u_-+\int_\Omega \lambda \beta(\tilde{u})u_-.
$$
This, together with $\lambda \beta(\tilde{u})\geq 0$, that $\lambda\beta(t)-g(t)$ is nondecreasing and that $\overset{\to}{f}=0$ on $(-\infty, 0)$, implies
$$
\|\nabla u_-\|^p_{L^p(\Omega)}\leq 0,
$$
which means that $u_-=0$ a.e. and then $u\geq 0$ a.e..

We now show that $u\leq k$ a.e.. In this case we take the test function $\varphi=(u-k)_+\in \W$ and use the fact that $\langle \qfrac u ,(u-k)_+ \rangle \geq 0$, then
$$
\int_\Omega (\lambda\beta(u)-\lambda \beta(\tilde{u}))(u-k)_++\|(u-k)_+\|^p_\Wpz\leq 0,
$$
and $(\lambda\beta(u)-\lambda \beta(\tilde{u}))(u-k)_+\geq 0$ because $\beta$ is nondecreasing, hence $u\leq k$.
\end{proof}

\begin{proof} [Proof of Theorem \ref{para2}]

\textbf{Step 1:} We first take $g$ Lipschitz such that $g=0$ on $\mathbb{R}\backslash (0,k)$. And we prove existence of a weak solution satisfying:

\begin{equation}\label{estimate}
\|\partial_t u\|^2_{L^2(Q_T)}+\|u\|^p_{L^\infty(0,T;\Wpz)}+\mu \|u\|_{L^\infty(0,T;\Wsqz)}^q \leq C.
\end{equation}

As the method is similar to the one of the proof of Theorem~\ref{para1}, we only highlight here the main differences. 

Using Theorem~\ref{elli}, we define the sequence $(u^n)_n\subset \W$ by
\begin{equation*}
    \begin{cases}
     \frac{1}{\dt} \beta(u^n)+ \A  u^n=g(u^n)+\Div (\overset{\to}{f}(u^n))+\frac{1}{\dt} \beta(u^{n-1}) \quad \text{in} \;\Omega,\\
      u^n=0 \quad \text{in} \; \Rn\backslash \Omega , \\
      u^n\in [0,k] \; a.e..
    \end{cases}
\end{equation*}
where $u^0=u_0$. Now, we can use the uniform bound given by the assumption on $g$, and by the same method, we get similar a priori estimates. In particular, we have that
\begin{equation}\label{2bis}
\|\partial_t \tilde{u}_\dt\|^2_{L^2(Q_T)}+\max_{n\in \llbracket 0, N \rrbracket} J_\A(u^n) \leq \tilde{C},
\end{equation}
for a positive constant $\tilde C$ that is independent of $\dt$ and the only dependence of if on $g$ is through $\|g\|_{L^\infty([0,k])}$ to a positive power.
%$\overset{\to}{f}$, $\|g\|_{L^\infty([0,k])}$, $p$, $d$, $\Omega$ and $T$.

The convergence part of the proof is similar, the only different point is the convergence of $(g(u_\dt))_\dt$, simply given now by the fact that $g$ is Lipschitz, thus for $\gamma\in [1,\infty)$
\begin{equation*}
g(u_\dt)\to g(u) \text{ in } L^\infty(0,T;L^\gamma(\Omega)) \text{ as  } \dt\to0.
\end{equation*}

We can pass to the limit in the equation in the same way to get a weak solution. Moreover, we obtain \eqref{estimate} by \eqref{2bis} and using \eqref{3} and \eqref{4}.

\textbf{Step 2:}

We now take $g$ Lipschitz and $g(0)=0$. We set a sequence of approximating functions $g_n$, that are Lipschitz and such that $g_n=g$ on $[0,k]$ and $g_n=0$ on $(-\infty,0]\cap[k+\frac{1}{n},\infty)$, and that $\|g_n\|_{L^\infty(\mathbb{R}^+)}\leq \|g\|_{L^\infty(0,k)}$ for all $n$.

By the Step 1, we obtain a sequence of solutions $(u_n)_n$ to each associated problem \eqref{pbp} where $g$ is replaced by $g_n$ and such that $u_n\in[0,k+\frac{1}{n}]$ for all $n$. Then, using \eqref{estimate}, we have up to a subsequence
\begin{itemize}
\item $ u_n \overset{*}{\rightharpoonup}u \text{ in }L^\infty(0,T;\Wpz)$,
\item $ \partial_t u_n \rightharpoonup \partial_t u \text{ in } L^2(Q_T)$.
\end{itemize}

Using again \eqref{estimate} and Aubin-Simon Lemma, we obtain
$$
u_n, \beta(u_n) \to u, \beta(u) \text{ a.e. and in } C([0,T];L^\gamma(\Omega)),\; \forall \gamma \in [1,\infty).
$$
Using the a.e. convergence we have that $u\in[0,k]$ a.e. and that, for all $\gamma \in [1,\infty)$,
$$
\overset{\to}{f}(u_n) \to \overset{\to}{f}(u) \text{ in } C([0,T];L^\gamma(\Omega)) , \; \forall \gamma \in [1,\infty).
$$

We now show convergence for $g_n(u_n)$. First, we write,
\begin{equation}\label{8}
g_n(u_n)=g_n(u_n)\mathbb{1}_{\{ u_n\in(k,k+\frac{1}{n})\}}+g(u_n)\mathbb{1}_{\{ u_n\in[0,k]\}}.
\end{equation}
Since $u_n\to u \text{ in } C([0,T];L^\gamma(\Omega)),\; \forall \gamma \in [1,\infty)$ and $g$ is Lipschitz, we have that $\forall \gamma \in [1,\infty)$,
\begin{equation}\label{9}
g(u_n) \to g(u) \text{ in } C([0,T];L^\gamma(\Omega)) \text{ as } n\to\infty.
\end{equation}

On the other hand, using that $\|g_n\|_{L^\infty(\mathbb{R}^+)}\leq \|g\|_{L^\infty(0,k)}$, we also have that for all $\gamma \in [1,\infty)$
\begin{equation}\label{10}
g_n(u_n)\mathbb{1}_{\{ u_n\in(k,k+\frac{1}{n})\}} \to 0 \text{ in } C([0,T];L^\gamma(\Omega)) \text{ as } n\to \infty.
\end{equation}

Combining \eqref{8}, \eqref{9} and \eqref{10} we obtain, for all $ \gamma \in [1,\infty)$,
$$
g_n(u_n) \to g(u) \text{ in } C([0,T];L^\gamma(\Omega)) \text{ as } n\to\infty.
$$

As for the Step 3 of the proof of Theorem~\ref{para1}, passing to the limit in operator terms and then in the weak formulation, we can conclude that $u\in [0,k]$ is a weak solution. 

Finally, using Remark~\ref{rk} and that $g(u)\in L^\infty(Q_\infty)$, we get that $u$ is a global weak-mild solution.
\end{proof}

\section{Qualitative behavior}\label{sec3}
In this section we prove some results on the global behavior of solutions, in particular we explore conditions under which stabilization to a non-trivial steady state, extinction and blow up occur.

\subsection{Convergence to the the steady state: Proof of Theorem \ref{thstab}}
We first study the following steady state problem, for $h\in L^\infty(\Omega)$ nonnegative
\begin{equation}\label{pbstab}
    \begin{cases}
 \A u=\Div (\overset{\to}{f}(u))+h\quad \text{in} \;\Omega,\\
      u=0 \quad \text{on} \; \Rn\backslash \Omega.
    \end{cases}\
\end{equation}

Following similar arguments as in the proof of Theorem \ref{elli2}, we can conclude that there exists a nontrivial bounded weak solution of problem \eqref{pbstab}. Also using the same method as in the proof of Theorem~\ref{elli2}, we have that the solutions are nonnegative. The following Lemma gives a comparison principle (here $\mu>0$).

\begin{lemma}\label{princcomp}
Let $\mu>0$ and $u$, $v\in\{ u\in \W \; |\; \A u-\Div \ff (u) \in L^1(\Omega)\}$, such that
\begin{equation}\label{29}
\A u-\Div \ff (u) \leq \A v-\Div \ff (v) \text{ a.e. in } \Omega,
\end{equation}
then, $u\leq v$ a.e. in $\Omega$.
\end{lemma}

\begin{proof}

Let $u,v\in \W$ as above, we first show that:
\begin{equation}\label{15}
\int_\Omega \Div(\ff(u)-\ff(v))\mathbb{1}_{\{u>v\}}=0.
\end{equation}
We set $\overline{f}_i$ and $\underline{f}_i$ defined as in the proof of Theorem \ref{thacc}. Observe that:
$$
\int_\Omega \Div(\overline f(u)-\overline f(v))\mathbb{1}_{\{u>v\}}=\int_\Omega \Div((\overline f(u)-\overline f(v))_+)=0.
$$
with $\overline{f}\in W^{1,\infty}(\Omega)$. Similarly, we obtain
$$\int_\Omega \Div(\underline f(u)-\underline f(v))\mathbb{1}_{\{u>v\}}=0,$$
thus, we have \eqref{15}. We now show that if
\begin{equation}\label{17}
\int_\Omega (\A u-\A v)\mathbb{1}_{\{u>v\}}\leq 0, 
\end{equation}
then $u\leq v$ a.e. in $\Omega$. In order to show this, we set
\begin{align*}
\gamma_\varepsilon(x)= \left\{ \begin{array}{lr} 1 \text{ if } x>\varepsilon, \\
\Psi_\varepsilon(x) \text{ if } x\in [0,\varepsilon], \\
0 \text{ if } x<0,
\end{array} \right.
\end{align*}
where $\Psi_\varepsilon\in C^2([0,\varepsilon])$, $\Psi(0)=0$, $\Psi_\varepsilon(\varepsilon)=1$ and $\Psi_\varepsilon'\geq 0$. Clearly, $\gamma_\varepsilon(u-v) \in L^\infty(\Omega)\cap \W$.

We now subtract the weak formulation of the equations for $u$ and $v$ and test with 
$\gamma_\varepsilon(u-v)$. We look at the integrands first, for which the following should be understood to hold $a.e.$. 

We use here the notation $\Phi_q(s-t)=|s-t|^{q-2}(s-t)$, then,
\begin{equation*}
\frac{\Phi_q(u(x)-u(y))-\Phi_q(v(x)-v(y))}{|x-y|^{N+sp}}(\gamma_\varepsilon(u(x)-v(x))-\gamma_\varepsilon(u(y)-v(y)))\geq 0,
\end{equation*}
indeed $\gamma_\varepsilon$ is nondecreasing. Let us argue now by contradiction, if there exists $K\subset \Omega$ such that $|K|>0$ and $u>v$ on $K$ then we have in $K$ that
\begin{equation*}
\frac{\Phi_q(u(x)-u(y))-\Phi_q(v(x)-v(y))}{|x-y|^{N+sp}}(\gamma_\varepsilon(u(x)-v(x))-\gamma_\varepsilon(u(y)-v(y)))\geq C>0.
\end{equation*}
This gives
\begin{equation*}
\langle\A u-\A v,\gamma_\varepsilon(u-v)\rangle\geq \langle\qfrac u-\qfrac v,\gamma_\varepsilon(u-v)\rangle\geq C>0.
\end{equation*}

Taking the limit $\varepsilon\to 0$, we have that $\gamma_\varepsilon(u-v)\to \mathbb{1}_{\{u>v\}}$ in $L^\infty(\Omega)$, and hence
$$
\int_\Omega (\A u-\A v)\mathbb{1}_{\{u>v\}}\geq C,
$$
but this contradicts \eqref{17}.
\end{proof}

Using the elliptic comparison principle on the discretization we get the following parabolic comparison principle.
\begin{corollary}\label{compara}
Let $\mu\geq 0,$ $u_0\leq v_0$ and $h_1,h_2\in L^\infty(Q_T)$ such that $h_1\leq h_2$. We note $u,v$ the weak-mild solutions of \eqref{pbp} for source term $g=h_1$ and $h_2$ respectively, then $u\leq v$.
\end{corollary}
\begin{proof}
Using the uniqueness of the mild solution we can look at the discretization. For $\lambda \beta(u)+\A u -\Div\ff(u)$ we still have a comparison principle using that $\beta$ is increasing. Also when $\mu=0$ the $\beta$ term has the same purpose as the $\qfrac$ term in the proof of Lemma \ref{princcomp} and we can conclude by an analogous the comparison principle. Finally, the proof of the corollary follows by applying this principle to the discretization in time by an induction argument.
\end{proof}

\begin{proof}[Proof of Theorem \ref{thstab}]
We first take $u_0=0$. With the hypothesis of Theorem \ref{thstab}, we have a global solution using Theorem~\ref{para1}. Using Corollary \ref{compara} for $u_0=0$ and $v_0=u_{stat}$ we have that $u\leq v $ where $v=u_{stat}$ is constant in time, similarly we have $u\geq 0$. 
Using $u_0=0$, we have, in the first step of the discretization scheme, that $u^1\geq u_0$. We can iterate the comparison to get, with Lemma \ref{princcomp} and an induction argument, that $u$ is nondecreasing in time. Now, by the monotone convergence theorem we have that there exists $u_\infty\in L^\gamma(\Omega)$ such that $u\to u_\infty$ in $L^\gamma(\Omega)$ for all $\gamma<\infty$ when $t\to \infty$. 

Using now for each $i=1,\dots,d$, $f_i\in C(\mathbb{R})$, we find, with the notation $F_i(t)=\int_0^tf_i(s)ds$, that
$$
\int_\Omega \Div \ff(u) u = -\int_\Omega \ff(u) .\nabla u=-\int_\Omega \Div(F(u))=0,
$$
which gives
$$
d_t\|B(u)\|_1+\|u\|^p_\Wpz+\mu\|u\|^q_\Wsqz=\int_\Omega h u.
$$
Using here that $u$ is nondecreasing in time, we have $d_t\|B(u)\|_1\geq 0$ and then $u(t)$ is bounded in $\W\cap L^\infty(\Omega)$.

Next we use the following notation; if $\tilde v\in \W\cap L^\infty(\Omega)$ we let $S(t,\tilde v)=v(t)$ be the solution of \eqref{pbp} for the initial value $\tilde v$. Using a similar method as in Step 2 of the proof of Theorem~\ref{para2}, we have that if $(\tilde v_n)_n$ is a bounded sequence in $L^\infty \cap \W$ and $\tilde v_n\to \tilde v$ in $L^1(\Omega)$ as $\to\infty$, then $S(t,\tilde v_n) \to S(t,\tilde v)$ in $L^1(\Omega)$ as $n\to\infty$, thus
$$
S(t,\lim_{n\to \infty} S(Tn,u_0))=S(t,{u}_\infty).
$$
Using how we have defined the global solution in the proof of Theorem~\ref{para1} we have,
$$
{u}_\infty=\lim_{n\to \infty} S(t+Tn,{u}_0)=\lim_{n\to \infty} S(t,S(Tn,{u}_0))=S(t,\lim_{n\to \infty} S(Tn,u_0))=S(t,{u}_\infty).
$$
This implies that $u_\infty=\tilde u_{stat}$, by uniqueness of the stationary solution, is a solution of \eqref{pbstab}. For $u_0\in [ 0,u_{stat}]$ we can use Corollary \ref{compara} in a similar way to conclude the proof.
%finally we have $u_\infty=u_{stat}$  by uniqueness of the stationary solution. For $u_0\in [ 0,u_{stat}]$ we can use Corollary \ref{compara} to conclude.
\end{proof}

\begin{remark}
We note that we can generalize Theorem \ref{thstab} by taking a suitable nonlinear source term. This is possible as long as a comparison principles and existence of a unique nontrivial solution of the elliptic problem can be guaranteed. For example, taking $g(x,u)$ a Caratheodory function such that $g(x,\cdot)$ is nonincreasing and $0\leq g(\cdot,0)\in L^\infty(\Omega) $ gives existence and uniqueness of the elliptic problem. Also using the discretization in tiome scheme
$$
\frac{\beta(u^n)-\beta(u^{n-1})}{\dt}+\A u^n=\Div \ff(u^n)+g(u^n),
$$
and $g$ nonincreasing we get a parabolic comparison principle. Finally, $g(x,0)\geq 0$ gives that $0$ is a subsolution of the associated stabilization problem thus we get stabilization as in the proof of Theorem \ref{thstab}.

We can also generalize the initial condition for the specific case $\ff =0$. We get using \cite{antonini} that if $sq<p$ then the solution of $\A u_K=|g(u_K)|+K$ satisfies $u_{K}\geq  c_K \text{dist}(\cdot,\Omega^c)$  thus by taking $u_0\in[0, u_K]$ we can use a sub and super solution method to get stabilization . And for the more specific case $q=p$ we have by homogeneity that $c_K\to \infty$ while $K\to \infty$ and so $u_0\leq C \text{dist}(\cdot,\Omega^c)$ implies the existence of a supersolution $u_K\geq u_0$. Indeed, the arguments employed in \cite{constantin2024existence}) can be adapted. 
\end{remark}

\subsection{Extinction and blow up: Proofs of Theorems \ref{thextin} and \ref{thexpl}}

In this section, we use energy methods to prove Theorems \ref{thextin} and \ref{thexpl}. We start with the proof of the former, where we obtain extinction results.

\begin{proof}[Proof of Theorem \ref{thextin}]
Note that with our hypotheses and Theorem \ref{para1} we have a global solution.
We take the test function $|u|^{k-1}u$ for $k\geq 1$, using that each $f_i\in C(\mathbb{R})$ with $i=1,\dots,d$ and using the notation $F_i(t)=\int_0^tf_i(\theta^\frac{1}{k})d\theta$, we get 
$$
\int_\Omega \Div \ff(u) |u|^{k-1}u \,dx = -\int_\Omega \ff(u) .\nabla (|u|^{k-1}u)\,dx=-\int_\Omega \Div(F(|u|^{k-1}u))\,dx=\int_\Gamma F(0) \,d\nu=0.
$$
This gives, together with $\langle \Delta_p u,|u|^{k-1}u\rangle\geq 0$, that
$$
\frac{m}{km+1}d_t\|u\|^{\frac{1}{m}+k}_{\frac{1}{m}+k}+\mu\langle \qfrac u, |u|^{k-1}u\rangle \leq \|u\|_{r+k}^{r+k}.
$$

From \eqref{ineg2} we infer that $(u(x)-u(y))(|u|^{k-1}u(x)-|u|^{k-1}u(y))\geq c|u(x)-u(y)|^{k+1}$, where the constant $c$ depends only on $k$ and $d$. Then,
\begin{equation*}
\langle \qfrac u,|u|^{k-1}u \rangle \geq c \int_\Rn \int_\Rn \frac{|u(x)-u(y)|^{q-1+k}}{|x-y|^{d+sq}}\,dxdy=c\|u\|^{\tilde{q}}_{W^{\tilde{s},\tilde{q}}_0}
\end{equation*}
where 
$$
\tilde{q}=q-1+k, \ \tilde{s}=\dfrac{sq}{\tilde{q}}.
$$ 

For $k$ big enough we have $W^{\tilde{s},\tilde{q}}_0(\Omega)\hookrightarrow L^{k+\frac{1}{m}}(\Omega)$, and thus we fix such a $k$, taking for example $k\geq \min(1,\frac{d-sq-d(q-1)m}{msq})$. This gives
$$
c_1d_t\|u\|^{\frac{1}{m}+k}_{\frac{1}{m}+k}+\frac{\mu}{2} \|u\|^{\tilde q}_{\frac{1}{m}+k}\leq c_2\|u\|_{\frac{1}{m}+k}^{r+k}-\frac{\mu}{2}\|u\|^{\tilde{q}}_{\frac{1}{m}+k}.
$$

We observe that, since $\tilde q < r+k$, there exists $c$ such that $\|u\|_{\frac{1}{m}+k}^{r+k}-\frac{\mu}{2}\|u\|^{\tilde q}_{\frac{1}{m}+k}\leq 0$, if $\|u\|_{\frac{1}{m}+k}\leq c$. Now, taking $\|u_0\|_{\frac{1}{m}+k}\leq c$ and using that $u\in C([0,T];L^{\frac{1}{m}+k}(\Omega))$ we arrive at the inequality
$$
c_1d_t\|u\|^{\frac{1}{m}+k}_{\frac{1}{m}+k}+\frac{\mu}{2}\|u\|^{\tilde q}_{\frac{1}{m}+k}\leq 0.
$$

We now write the last estimate by setting $Y(t)=\|u\|^{\frac{1}{m}+k}_{\frac{1}{m}+k}$, for simplicity of notation. Hence, since $Y(t)>0$ for all $t$, we have, for some $C>0$ independent of $t$, that
$$
\frac{d_t Y(t)^{1-\alpha}}{1-\alpha}+C\leq 0 \text{ with } \alpha=\frac{q-1+k}{\frac{1}{m}+k}<1.
$$
%where $\alpha=\frac{q-1+k}{\frac{1}{m}+k}<1$ and 
Integrating with respect to $t$ gives
$$
Y(t)^{1-\alpha}\leq -Ct+Y(0)^{1-\alpha},
$$
and then $Y(t)<0$ for $t$ large enough. But this gives a contradiction, then $Y(t)$ must tend to $0$ as $t$ tends to a finite value smaller than or equal to  $Y(0)^{1-\alpha}/C$.
\end{proof}

\begin{proof}[Proof of Theorem \ref{thexpl}]
First, we note that with the condition \eqref{condf}, the Hölder inequality and Sobolev embedding results we have the estimate
\begin{equation}\label{37}
\| \Div \ff(u) \| ^2_2 \leq C\left(\| \nabla u\|^2_p+\|\nabla u \|^{2(\gamma+1)}_p\right).
\end{equation}

We first consider the case $q\leq p<r+ 1$. In this case we can choose a value $\varepsilon\in (0,1)$ such that 
$$
c_\varepsilon:=\left(1-\frac{p}{(r+1)(1-\varepsilon)}\right)>0.
$$
With this we now define the following functional
$$
I(u):=\frac{1-\varepsilon}{p}\|\nabla u \| ^p_p-\frac{1}{r+1} \| u \| ^{r+1}_{r+1}+\mu \frac{1-\varepsilon}{p}\| u \| _\Wsqz^q.
$$
Using the discretization in time scheme, we have as for \eqref{26} (using \eqref{41} and passing to the limit $\dt\to 0$) that for a.e. in $t$:
\begin{equation*}
 J_\A(u(t))-\frac{ \| u(t) \| ^{r+1}_{r+1} }{r+1}\leq J_\A(u_0)+C\| \Div \ff( u)\|^2_{L^2(Q_t)}-\frac{\| u_0 \| ^{r+1}_{r+1}}{r+1} ,
\end{equation*}
which, together with \eqref{37}, immplies that
\begin{equation}\label{33}
\begin{split}
 \frac{\varepsilon}{p}\| \nabla u\|^p_p +\mu \left(\frac{1}{q}-\frac{1-\varepsilon}{p}\right) &  \| u\| ^q _\Wsqz	 +I(u) \\
 &\leq E(u_0)+C\left(\| \nabla u\|^2_{L^p(Q_t)}+\|\nabla u \|^{2(\gamma+1)}_{L^p(Q_t)}\right).
 \end{split}
\end{equation}
Now, by integrating \eqref{33}, we have that for all $t$:
$$
\int_0^tI(u(s))ds \leq t\left(E(u_0)+C\left(\| \nabla u\|^2_{L^p(Q_t)}+\|\nabla u \|^{2(\gamma+1)}_{L^p(Q_t)}\right)\right)-\frac{\varepsilon}{p} \|\nabla u\| ^p_{L^p(Q_t)}.
$$

Now we set $T^*= c_0\| u_0\| ^{\frac{1}{m}-r}_{\frac{1}{m}+1}$ where $c_0=\frac{1+m}{c_1(mr-1)}$, with $c_1= c_\varepsilon({\frac{1}{m}+1})c_r$ and $c_r$ is the embedding constant. We assume for the moment existence of a weak-mild solution in $(0,T^*]$. We have, using \eqref{ineqener}, that for all $t\leq T^*$:
$$
\int_0^t I(u(s)) ds \leq 0.
$$
Now, taking the test function $u$ in the equation, we get, for a.e. $t$, that
$$
\frac{m}{m+1}d_t\|u\|^{\frac{1}{m}+1}_{\frac{1}{m}+1}=c_\varepsilon \|u\|^{r+1}_{r+1}-\frac{p}{1-\varepsilon}I(u),
$$
and, then, for all $t\leq T^*$ we have also that
$$
\|u(t)\|^{\frac{1}{m}+1}_{\frac{1}{m}+1}\geq \|u_0\|^{\frac{1}{m}-1}_{\frac{1}{m}+1}+c_1\int_0^t\|u(t)\|^{r+1}_{\frac{1}{m}+1}
$$
which gives blow up of the quantity $\|u(t)\|_{\frac{1}{m}+1}$ at a finite value of $t$ smaller than or equal to $T^*$.

For the case $r+1>q>p$ we have can obtain the result by setting the functional
$$
I(u)=\frac{\mu}{q}\| u\| ^q _\Wsqz+\frac{1}{q}\| \nabla u\|^p_p-\frac{1}{r+1}\| u\| ^{r+1}_{r+1}
$$
and arguing similarly, since now we have the estimate,
$$
\left(\frac{1}{p}-\frac{1}{q}\right)\|\nabla u\|^p_p-C(\| \nabla u\|^2_{L^p(Q_t)}+\|\nabla u \|^{2(\gamma+1)}_{L^p(Q_t)})+I(u) \leq E(u_0).
$$
\end{proof}

\section*{Acknowledgement} We thank Jaques Giacomoni for helpful discussions on the topic.

\section*{Funding sources} Loïc Constantin was supported by a fully funded public doctoral contract, provided by the University of Pau (Université de Pau et des Pays de l’Adour – UPPA), the funding comes entirely from a French public institution under the 2023 UPPA/UPV cross-border call for projects (AAP UPPA/UPV 2023).  Carlota M. Cuesta was partially supported by the research
group grant IT1615-22 funded by the Basque Government, and by the grant PID2021-126813NB-I00 funded
by MICIU/AEI/10.13039/501100011033 and by “ERDF A way of making Europe”. 

\appendix
\section{Appendix:}

In this section we list some asuxiliary results that are needed in the proofs of the main results.

We start with the basic inequalities appropriate for our nonlinear operator. Using $p\geq 2$ we have with \cite{simon},
\begin{propriete}\label{inegalg}
There exist $c_1$, $c_2$ positive constants such that for all $\xi , \eta \in \mathbb{R}^d$:

\begin{align}\label{ineg1}\tag{A1}
||\xi|^{p-2}\xi -|\eta|^{p-2}\eta|\leq c_1 \begin{array}{lr}  |\xi -\eta |(|\xi|+|\eta|)^{p-2}, \\
%|\xi-\eta|^{p-1} \text{ if },
\end{array} 
\end{align}

\begin{align}\label{ineg2}\tag{A2}
(|\xi|^{p-2}\xi -|\eta|^{p-2}\eta).(\xi-\eta) \geq c_2  \begin{array}{lr}  |\xi -\eta |^{p}. \\
%\frac{|\xi-\eta|^{2}}{(|\xi|+|\eta|)^{2-p}}.
\end{array}
\end{align}

\end{propriete}

We also need the next integration by parts.
\begin{lemma}\label{ipp}\cite[Lemma 7.3, p.191]{roubi}

We have the continuous embedding $H^1(0,T;L^2(\Omega)) \hookrightarrow C([0,T];L^2(\Omega))$, and for all $u,v\in H^1(0,T;L^2(\Omega))$ and all $t\in[0,T]$:

\begin{equation*}
\Big[\int_\Omega uv\Big]^t_0=\int_0^t \int_\Omega( u\partial _t v+v\partial _t u).
\end{equation*}

Where $[a(s)]^t_0=a(t)-a(0)$.

\end{lemma}

The next is a chain rule for a Lipschitz function composed with a $W^{1,p}$ function.
\begin{theorem}\cite[Th 2.1.11]{ziemer89}\label{chain}
Let $f:\mathbb{R}\to \mathbb{R}$ be a Lipschitz function and $u\in \Wp, p\geq 1.$ If $f\circ u \in L^p(\Omega),$ then $f\circ u \in \Wp$ and for almost all $x\in\Omega,$
$$D(f\circ u )(x)=f'(u(x))Du(x).$$

\end{theorem}

We finish this appendix with the definition of an $\varepsilon$-discretization needed for the definition of a mild solution as given in \cite{barbu2}.
\begin{definition}[%$\varepsilon$-discretization and
$\varepsilon$-approximate solution]\label{defapp}
Let $f\in L^1(Q_T), \varepsilon>0$ and  $(t_i)_{i\in \{1,...,N\}}$ be a partition of $[0,T]$ such that for any $i$, $t_i-t_{i-1}<\varepsilon$.\\
An $\varepsilon$-approximate solution of \eqref{pbo} is a piecewise constant function $U:[0,T]\to L^1(\Omega)$ defined by $U(t)=U_i=U(t_{i})$ on $[t_i,t_{i+1})$ such that:
$$\|U(0)-u_0\|_{L^1(\Omega)}\leq\varepsilon$$
and
$$\frac{U_i-U_{i-1}}{t_i-t_{i-1}}+\pfrac (|U_i|^{m-1} U_i)=f_i \quad \mbox{on } [t_i,t_{i+1}),$$
where $(f_i)_{i\in \{1,...,N\}}$ satisfies $ \displaystyle\sum_{i=1}^N\int_{t_{i-1}}^{t_i}\|f(\tau)-f_i\|_{L^1(\Omega)}d\tau <\varepsilon$.

\end{definition}

\bibliographystyle{plain}
\bibliography{biblio}

\end{document}